\documentclass[reqno]{amsart}
\usepackage{amssymb}
\usepackage{graphicx}
\usepackage{amsmath}
\usepackage{mathbbol}

\usepackage[usenames, dvipsnames]{color}
\usepackage{verbatim}
\usepackage{mathrsfs}
\usepackage{bm}
\usepackage{cite}
\usepackage{color}
\allowdisplaybreaks[4]

\numberwithin{equation}{section}

\newtheorem{theorem}{Theorem}[section]

\newtheorem{lemma}[theorem]{Lemma}
\newtheorem{prop}[theorem]{Proposition}

\theoremstyle{definition}
\newtheorem{remark}[theorem]{Remark}

\theoremstyle{definition}

\theoremstyle{definition}

\makeatletter
\def\dashint{\operatorname%
{\,\,\text{\bf-}\kern-.98em\DOTSI\intop\ilimits@\!\!}}
\makeatother

\def\\det{\text{\det}}

\def\.5{\frac{1}{2}}

\def\h{\mathcal{H}}

\newcommand{\RN}[1]{%
  \textup{\uppercase\expandafter{\romannumeral#1}}%
}

\renewcommand{\epsilon}{\varepsilon}

\newcounter{marnote}

\begin{document}
	\title[Boundary stress blow-up in 2D Stokes flow]{Stress blow-up analysis when a  suspending\\ rigid particle approaches the
boundary\\ in Stokes flow:  2D case}
	
	\author[H.G. Li]{Haigang Li}
	\address[H.G. Li]{School of Mathematical Sciences, Beijing Normal University, Laboratory of MathematiCs and Complex Systems, Ministry of Education, Beijing 100875, China.}
	\email{hgli@bnu.edu.cn}
	
	\author[L.J. Xu]{Longjuan Xu}
	\address[L.J. Xu]{Academy for Multidisciplinary Studies, Capital Normal University, Beijing 100048, China.}
	\email{ljxu311@163.com}
	
	\author[P.H. Zhang]{PeiHao Zhang}
	\address[P.H. Zhang]{School of Mathematical Sciences, Beijing Normal University, Beijing 100875, China.}
	\email{202021130034@mail.bnu.edu.cn}

	%\footnote{}

	\date{\today} % delete this line to display the current date
	
	%%% BEGIN DOCUMENT
\begin{abstract}
It is an interesting and important topic to study the motion of small particles in a viscous liquid in current applied research. In this paper we assume the particles are convex with arbitrary shapes and mainly investigate the interaction between the rigid particles and the domain boundary when the distance tends to zero. In fact, even though the domain and the prescribed boundary data are both smooth, it is possible to cause a definite increase of the blow-up rate of the stress. This problem has the free boundary value feature due to the rigidity assumption on the particle. We find that the prescribed local boundary data directly affects on the free boundary value on the particle. Two kinds of boundary data are considered: locally constant boundary data and locally polynomial boundary data.  For the former we prove the free boundary value is close to the prescribed constant, while for the latter we show the influence on the blow-up rate from the order of growth of the prescribed polynomial. Based on pointwise upper bounds in the neck region and lower bounds at the midpoint of the shortest line between the particle and the domain boundary, we show that these blow-up rates obtained in this paper are optimal. These precise estimates will help us understand the underlying mechanism of the hydrodynamic interactions in fluid particle model.
\end{abstract}
	
	\maketitle
	\tableofcontents
	
\section{Introduction}
	
\subsection{Background}\label{subsec1.1}

The study of the motion of small particles in a viscous liquid is one of the main focuses of current applied research. The particles interact with each other through the fluid flow since the movement of one particle also changes the fluid flow at the other particles. The presence of the particles affects the flow of the liquid, and this, in turn, affects the motion of the particles, so that this makes fundamental mathematical problem related to liquid-particle interaction a particularly challenging one. In fact, the mathematical theory of the motion of rigid particles in a liquid is one of the oldest and most classical problems in fluid mechanics, owed to the seminal contributions of Stokes \cite{Stokes}, Kirchhoff \cite{Kirchhoff}, and Jeffery \cite{Jeffery}. In recent two decades there appears a systematic study of the basic problems related to liquid-particle interaction by mathematicians, see \cite{Fortes,Weinberger,Galdi98,Galdi00,Desjardins,Hoffmann,Gunzburger,Martin} and the references therein. The interaction between particle pair and that between particle and the domain boundary are fundamental mechanisms that enter strongly in all practical applications of particulate flows \cite{DDJoseph}. The existence of global weak solutions for the two-dimensional motion of several rigid bodies in an incompressible viscous fluid, governed by the Navier-Stokes equations, is demonstrated by San Mart\'in et al \cite{SanMartin}, where they do not assume the lack of collisions. However, as a matter of fact, the stress in the neck region between rigid bodies or between a rigid body and the domain boundary may become arbitrarily large when the distances there between them tend to zero.

Since the Reynolds number usually can be very small in experiments, for the sake of simplicity in this paper we consider its Stokes approximation, the limiting situation of vanishingly small Reynolds number. Thus, the stress due to viscosity is predominant on that due to inertia, and so that inertial effects can be neglected. It is important to understand microstructure in flowing suspensions of spherical bodies. Very recently, for two $\varepsilon$-apart rigid circular inclusions suspending in the two-dimensional incompressible Stokes flow 
$$\mu\,\Delta {\bf u}=\nabla p,\quad\nabla\cdot {\bf u}=0,$$
Ammari, Kang, Kim, and Yu \cite{AKKY} made use of the bipolar coordinates to derive an asymptotic representation formula for the stress and showed the blow-up rates of the gradient and the Cauchy stress are  both $\varepsilon^{-1/2}$ in dimension two. The first two authors \cite{LX} extended to study two adjacent rigid convex inclusions of arbitrary shapes in dimensions two and three, by adapting the energy iteration technique developed in \cite{LLBY, BLL,BLL2} for dealing with the stress concentration problem in high-contrast elastic composite material. In \cite{LX} they established the pointwise upper bounds of the gradient for locally symmetric inclusions, and a lower bound at the narrowest place between inclusions under certain additional symmetric assumption on the boundary data, which shows that the optimal blow-up rate of the gradient is $(\varepsilon|\ln\varepsilon|)^{-1}$  in dimension three, while that of the Cauchy stress tensor is $\varepsilon^{-1}$. These results in  \cite{AKKY,LX} can be regarded as interior estimates on the interaction
between two close-spaced inclusions.

As a continuation of \cite{LX}, in this paper we consider the corresponding boundary gradient blow-up and study the interaction between the rigid particles and the domain boundary. In fact, even though the domain and the prescribed boundary data are both smooth, it is possible to cause a definite increase of the blow-up rate of the gradient. Firstly, thanks to the rigidity assumption on the particle, the problem has the free boundary value feature. The direct effect from the prescribed local boundary data on the free boundary value on the particle is discovered, and two kinds of boundary data are investigated: locally constant boundary data and locally polynomial boundary data.  For the former we prove the free boundary value is close to the the prescribed constant, while for the latter we show the influence on the blow-up rate from the order of growth of the prescribed polynomial. Secondly, although the main idea is from \cite{LX}, the asymmetry of the domain causes a lot of technical difficulties in the process of the construction of a series of divergence free auxiliary functions for the velocity vector, which will be proved to have the main singular terms in their gradients. More to the point, a series of associated  pressure auxiliary functions are constructed to capture the main singularity of the pressure terms, without solving the stream function. Here each auxiliary function is more complicated than that in the linear elasticity problem, and the computation is troublesome. Finally, the blow-up rates obtained in this paper are optimal, based on pointwise upper bounds in the neck region and lower bounds at the midpoint of the shortest line between the particle and the domain boundary, with new blow-up factors found due to the complicated effect from the different boundary data.  These precise estimates established in this paper help us understand the underlying mechanism of the hydrodynamic interactions in fluid particle model.

We would like to remark that one of the motivations of this paper is from the study of the stress concentration problem and electric field concentration problem in high-contrast composite materials. The concentrated stress field in fiber-reinforced composites may cause  failure of the composites \cite{Bab}, and the electric field can be greatly enhanced and utilized to achieve subwavelength imaging and sensitive spectroscopy \cite{YuAmmari}. Hence, it is quite important to understand the field concentration in a quantitatively precise manner. In this respect, since the numerical simulation work of Babu\u{s}ka et al. \cite{Bab}, there has been a lot of important literature in understanding this kind of field enhancement phenomenon. In the context of the perfect conductivity when the conductivity of two adjacent inclusions degenerates to $\infty$, the blow-up rate of the electric field is proved to be $\varepsilon^{-1/2}$ in dimension two \cite{AKL,Yun}, $(\varepsilon|\ln\varepsilon|)^{-1}$ in dimension three \cite{BLY,LY}, where $\varepsilon$ is the distance between two inclusions. While, the study for the insulated inclusions can refer to \cite{LiYang,Ben,DongLiYang}. For the Lam\'{e} system in linear elasticity, there is a significant difficulty in applying the methods for scalar equations to systems of equations that the maximum principle does not hold for the system. The first author, collaborated with Bao and Li, developed an energy iteration technique to  overcome this difficulty and proved the optimal blow-up rates of the gradient are $\varepsilon^{-1/2}$ in dimension two and $(\varepsilon|\ln\varepsilon|)^{-1}$ in dimension three, by providing pointwise upper bounds \cite{BLL,BLL2} and lower bound estimates \cite{Li2021}. The corresponding boundary estimates were as well studied in \cite{BJL}. The quantitative characterization and the asymptotic formula for the stress concentration was further  investigated in \cite{KY} by introducing singular functions in two dimensions and in \cite{LX0} for arbitrary convex inclusions in dimension three.  There is a long list of literature for more related work, see \cite{AKL3,ABTV,BC,BT,BV,DongLi,DongZhang,KLeY,KLiY,KLiY2,KangYu,Yun2} and the references therein. We end this subsection by mentioning that there is another interesting topic to investigate the sedimentation of a cloud of rigid particles under gravity in a Stokes fluid \cite{Hofer,Weinberger2}, where the existence of the solution was established in dimension two,  and the interactions between particles are an important issue there as well.

\subsection{Problem Formulation }\label{Subsec1.2}

Before we state our main results precisely, we first describe our domain and notations. For the sake of simplicity of computation, we first consider the problem in dimension two in this paper. Let $D\subset\mathbb{R}^{2}$ be a bounded open set, and  $D_1^0$ be a convex subset of $ D$, touching $\partial D$ at a point. We assume that $\partial D$ and $\partial D_1^0$ are both $C^{3}$ and their $C^{3}$ norms are bounded by some constant, independent of $\varepsilon$, to avoid their volumes tend to zero. By a translation and rotation of coordinates if necessary, we assume that $\partial D_1^0\cap\partial D=\{0\}$ and let the plane $\{x_{2}=0\}$ be their common tangent plane.  By a translating along the $x_{2}$-axis, set $D_1^{\varepsilon}:=D_{1}^{0}+(0,\varepsilon)$, and 
for simplicity of notation in the sequel, we drop the superscript $\varepsilon$ and denote $D_{1}:=D_{1}^{\varepsilon}$. 
Then 
\begin{equation*}%\label{assump D1 }
	\begin{split}
	\overline{D}_{1}\subset D,\quad
	\varepsilon:=\mbox{dist}(D_{1},\partial D)>0,\quad\text{and}\quad  \Omega:=D\setminus\overline{D}_{1}.
	\end{split}
\end{equation*}
Now we fix a small constant $0< R< 1$, independent of $\varepsilon$, such that the portions of $\partial D_1$ and $\partial D$ can be represented, respectively, by 
\begin{equation}\label{hh1'}
	 x_2=\varepsilon+h_1(x_1)\quad\text{and}\quad x_2=h(x_1),\quad \text{for}~ |x_1|\leq 2R,
\end{equation}
	 where $h_1$, $h\in C^{3}(B'_{2R}(0'))$ and satisfy 
\begin{align}
	 &\varepsilon+h_{1}(x_1)>h(x_1),\quad\mbox{for}~~ |x_1|\leq 2R,\label{h-h1}\\
	 &h_{1}(0')=h(0')=0,\quad  h'_{1}(0')=h'(0')=0,\label{h1h1}\\
	 &h_{1}(x_1)-h(x_1)=\kappa_0|x_1|^{2}+O(|x_1|^{3}),\quad\mbox{for}~~|x_1|<2R,\label{h1-h1}
\end{align}
where the constant $\kappa_0>0$. For $0\leq r\leq 2R$, let us define the neck region between $\partial D_{1}$ and $\partial D$ by
\begin{equation}\label{narrow-region}
	 \Omega_r:=\left\{(x_1,x_{2})\in \Omega:~ h(x_1)<x_{2}<\varepsilon+h_1(x_1),~|x_1|<r\right\}.
\end{equation}

Denote the linear space of rigid displacements in $\mathbb{R}^{2}$
	$$\Psi:=\Big\{{\boldsymbol\psi}\in C^{1}(\mathbb{R}^{2};\mathbb{R}^{2})~|~e({\boldsymbol\psi}):=\frac{1}{2}(\nabla{\boldsymbol\psi}+(\nabla{\boldsymbol\psi})^{\mathrm{T}})=0\Big\},$$
	with a basis $\{{\boldsymbol\psi}_{\alpha}\}_{\alpha=1,2,3}\in\Psi$: $${\boldsymbol\psi}_{1}=(1,0)^{\mathrm T},\quad{\boldsymbol\psi}_{2}=(0,1)^{\mathrm T},\quad{\boldsymbol\psi}_{3}=(x_{2},-x_{1})^{\mathrm T}.$$
We consider the following Stokes flow containing one rigid particle close to boundary:
\begin{align}\label{sto}
	\begin{cases}
	\mu \Delta {\bf u}=\nabla p,\quad\quad~~\nabla\cdot {\bf u}=0,&\hbox{in}~~\Omega:=D\setminus\overline{D}_{1},\\
	{\bf u}|_{+}={\bf u}|_{-},&\hbox{on}~~\partial{D_{1}},\\
	e({\bf u})=0, &\hbox{in}~~D_{1},\\
	\int_{\partial{D}_{1}}\frac{\partial {\bf u}}{\partial \nu}\Big|_{+}\cdot{\boldsymbol\psi}_{\alpha}-\int_{\partial{D}_{1}}p\,{\boldsymbol\psi}_{\alpha}\cdot
	\nu=0, &\alpha=1,2,3,
	\\
	{\bf u}={\boldsymbol\varphi}, &\hbox{on}~~\partial{D},
	\end{cases}
\end{align}
where $\mu>0$, 
$\frac{\partial {\bf u}}{\partial \nu}\big|_{+}:=\mu(\nabla {\bf u}+(\nabla {\bf u})^{\mathrm{T}})\nu$, $\nu$ is the  unit outer normal vector of $D_{1}$, and the subscript $+$ indicates the limit from inside the domain. Furthermore, via Gauss theorem and $\nabla\cdot{\bf u}=0$ in $\Omega$, one can verify that the prescribed velocity field ${\boldsymbol\varphi}$ satisfies the compatibility condition:
\begin{equation}\label{compatibility}
\int_{\partial{D}}{\boldsymbol\varphi}\cdot \nu\,=0,
\end{equation}
which guarantees the existence and uniqueness of the solution. Indeed, let
\begin{equation}\label{defPhi}
{\bf\Phi}:=\left\{\boldsymbol\varphi\in C^{1,\alpha}(\partial D;\mathbb R^2)~|~\boldsymbol\varphi~\mbox{satisfies}~ \eqref{compatibility}\right\},
\end{equation}
then for any given boundary ${\boldsymbol\varphi}\in{\bf\Phi}$ the existence and regularity of the solution now is standard, see \cite{ADN,Solo,GaldiBook}.

\subsection{The Main Difficulties and the Decomposition of the Solution}\label{subsec1.3}

Although it is known that there exists a pair solution $({\bf u},p)$ of \eqref{sto} and \eqref{compatibility} in $H^{1}(D;\mathbb{R}^{2})\times L^{2}(D)$,  the gradient of ${\bf u}$ and the press $p$ may blow up in the narrow region $\Omega_R$ when $\varepsilon$ tends to zero. The aim of the paper is to derive pointwise estimates of $\nabla{\bf u}$ and $p$ in $\Omega_R$ and to characterize such concentration phenomenon, especially affected from the prescribed boundary data. 

We will follow the method used in \cite{LX}. Introduce the Cauchy stress tensor
\begin{equation*}
	 	\sigma[{\bf u},p]=2\mu e({\bf u})-p\mathbb{I},
\end{equation*}
	 where $\mathbb{I}$ is the identity matrix. Then we reformulate \eqref{sto} as
\begin{align}\label{Stokessys}
	 \begin{cases}
	 	\nabla\cdot\sigma[{\bf u},p]=0,~~\nabla\cdot {\bf u}=0,&\hbox{in}~~\Omega,\\
	 	{\bf u}|_{+}={\bf u}|_{-},&\hbox{on}~~\partial{D_{1}},\\
	 	e({\bf u})=0, &\hbox{in}~~D_{1},\\
	 	\int_{\partial{D}_{1}}\sigma[{\bf u},p]\cdot{\boldsymbol\psi}_{\alpha}
	 	\nu=0
	 	,&\alpha=1,2,3,\\
	 	{\bf u}={\boldsymbol\varphi}, &\hbox{on}~~\partial{D}.
	 \end{cases}
\end{align}
Similarly as in \cite{BJL,LX}, by virtue of $e({\bf u})=0$ in $D_{1}$, we have
\begin{equation}\label{introC}
{\bf u}=\sum_{\alpha=1}^{3}C^{\alpha}{\boldsymbol\psi}_{\alpha}\quad\mbox{in}~D_1,
\end{equation} 
where the three free constants
$C^\alpha$, $\alpha=1,2,3$, will be determined by $u$ through the fourth line of \eqref {Stokessys}. By the continuity of the transmission condition (or the no-slip condition) on $\partial{D}_{1}$,  the solution of \eqref {Stokessys} in $\Omega$ can be decomposed as follows:
\begin{align}\label{udecom}
{\bf u}(x)=\sum_{\alpha=1}^{3}C^{\alpha}{\bf u}_{\alpha}(x)+{\bf u}_{0}(x)\quad\mbox{and}\quad p(x)=\sum_{\alpha=1}^{3}C^{\alpha}p_{\alpha}(x)+p_{0}(x),
\end{align}
where $({\bf u}_{\alpha},p_{\alpha})$, $\alpha=1,2,3$, satisfy
\begin{equation}\label{equ_v1}
\begin{cases}
\nabla\cdot\sigma[{\bf u}_\alpha,p_{\alpha}]=0,\quad\nabla\cdot {\bf u}_{\alpha}=0&\mathrm{in}~\Omega,\\
{\bf u}_{\alpha}={\boldsymbol\psi}_{\alpha}&\mathrm{on}~\partial{D}_{1},\\
{\bf u}_{\alpha}=0&\mathrm{on}~\partial{D},
\end{cases}
\end{equation}
and $({\bf u}_{0}, p_0)$ satisfies
\begin{equation}\label{equ_v3}
\begin{cases}
\nabla\cdot\sigma[{\bf u}_{0},p_0]=0,\quad\nabla\cdot {\bf u}_{0}=0&\mathrm{in}~\Omega,\\
{\bf u}_{0}=0&\mathrm{on}~\partial{{D}_{1}},\\
{\bf u}_{0}={\boldsymbol\varphi}&\mathrm{on}~\partial{D}.
\end{cases}
\end{equation}
Therefore,
\begin{equation}\label{equ_v4}
\nabla{\bf u}=\sum_{\alpha=1}^{3}C^\alpha\nabla {\bf u}_\alpha+\nabla {\bf u}_0\quad\mathrm{in}~\Omega.
\end{equation}
To estimate $\nabla{\bf u}$ and $p$, it suffices to derive the following three ingredients: (i) estimates of $|\nabla \bf u_\alpha|$ and $p_\alpha$, $\alpha=1,2,3$; (ii) estimate of $|\nabla{\bf u}_0|$ and $p_{0}$; and (iii) estimates of $|C^\alpha|$, $\alpha=1,2,3$. 

Compared with the derivation process for the interior estimates presented in \cite{LX}, new difficulties will be encountered in the above each ingredient. First, we assume that the two close-to-touching inclusions are locally symmetric near the origin in \cite{LX}. However, for the problem \eqref{sto}, here the boundary function $h(x_{1})$ and $h_{1}(x_{1})$ are not symmetric any more, which causes it troublesome to construct auxiliary functions satisfying the incompressible condition. This is the first difficulty we need to overcome in this paper. So that the auxiliary functions constructed in this paper looks more complicated than before. Second, it is obvious that the estimates of $|\nabla{\bf u}_0|$ and $p_{0}$ are directly affected from the choice of the boundary data. Because the blow-up may occur near the origin, by the Taylor expansion of ${\boldsymbol\varphi}(x)$ at the origin, we first consider the case that ${\boldsymbol\varphi}=const.$ and then consider the polynomial cases. The latter may result in the growth of singularity of $|\nabla{\bf u}_0|$  due to the effect from given ${\boldsymbol\varphi}$.

Finally, denote
\begin{align}\label{aijbj}
a_{\alpha\beta}:=-
\int_{\partial D_1}{\boldsymbol\psi}_\beta\cdot\sigma[{\bf u}_{\alpha},p_{\alpha}]\nu,\quad
Q_{\beta}[{\boldsymbol\varphi}]:=
\int_{\partial D_1}{\boldsymbol\psi}_\beta\cdot\sigma[{\bf u}_{0},p_{0}]\nu,
\end{align}
then combining the fourth line of \eqref{Stokessys} with decomposition \eqref{equ_v4} leads to 
\begin{equation}\label{ce}
\sum\limits_{\alpha=1}^{3}C^{\alpha}a_{\alpha\beta}=Q_{\beta}[{\boldsymbol\varphi}].
\end{equation}
To solve $C^{\alpha}$, we need to show the matrix $(a_{\alpha\beta})_{3\times3}$ is invertible and to give a positive lower bound of its determinant. By using the integration by parts, we have 
\begin{align}\label{defaij}
a_{\alpha\beta}=\int_{\Omega} \left(2\mu e({\bf u}_{\alpha}), e({\bf u}_{\beta})\right)\mathrm{d}x,\quad
Q_{\beta}[{\boldsymbol\varphi}]=-\int_{\Omega} \left(2\mu e({\bf u}_{0}),e({\bf u}_{\beta})\right)\mathrm{d}x.
\end{align}
So the estimates of $a_{\alpha\beta}$ totally depend on the estimates  of $|\nabla{\bf u}_{\alpha}|$, $\alpha=1,2,3$. Here we have to derive the asymptotic formula for each $a_{\alpha\beta}$ once it is an infinity to solve \eqref{ce}, see Lemma \ref{lema114D} below. On the other hand, it is obvious that from the definitions of $Q_{\beta}[{\boldsymbol\varphi}]$ that they depend on the choice of ${\boldsymbol\varphi}$ as well, so do the estimates of $C^{\alpha}$.  Therefore, in the sequel, we shall divide into two parts according to the choice of ${\boldsymbol\varphi}$ to investigate the influence on the blow-up of $\nabla{\bf u}$ and $p$ from ${\boldsymbol\varphi}$ and to show the optimality of these blow-up rates.

Because the stress concentration usually occurs near the origin, let us denote the part of $\partial{D}$ there by
$$\Gamma_{2R}=\left\{(x_1,x_{2})\in \Omega:~ x_{2}=h(x_1),~|x_1|<2R\right\}.$$
We will classify by the value of ${\boldsymbol\varphi}$ on $\Gamma_{2R}$ to investigate the role of ${\boldsymbol\varphi}$ in the blow-up analysis. Throughout this paper, we say a constant is {\em{universal}} if it depends only on $\mu,\kappa,\kappa_1$, and the upper bounds of the $C^{3}$ norm of $\partial{D}$ and $\partial{D}_{1}$, but independent of $\varepsilon$, where $\kappa$ and $\kappa_1$ are two contants independent of $\varepsilon$; see  \eqref{55} below.

\subsection{Main Results for Locally Constant Boundary Data Case}\label{subsec1.4}

 We first define the function class with constant-valued on $\Gamma_{2R}$: 
\begin{equation}\label{defphi1}
{\bf\Phi}_{1}:=\left\{{\boldsymbol\varphi}\in{\bf\Phi}~|~{\boldsymbol\varphi}=(1,0)^{\mathrm T}~\mbox{on}~\Gamma_{2R}\right\},
\end{equation}
and
\begin{equation}\label{defphi2}
{\bf\Phi}_{2}:=\left\{\boldsymbol\varphi\in{\bf\Phi}~|~{\boldsymbol\varphi}=(0,1)^{\mathrm T}~\mbox{on}~\Gamma_{2R}\right\},
\end{equation}
with their element denoted by $\boldsymbol\varphi_{i}$, $i=1,2$, where ${\bf\Phi}$ is defined in \eqref{defPhi}. Then, for instance, when ${\boldsymbol{\varphi}}\in{\bf\Phi}_{1}$ on $\Gamma_{R}$, then roughly speaking, ${\bf u}_{0}$ will be very close to ${\bf e}_{1}-{\bf u}_{1}$ in $\Omega_{R}$, which further leads to the free constant $C^{1}$ is close to $1$ as well. This reflects the effects on ${\bf u}$ from the choice of ${\boldsymbol\varphi}$ directly and indirectly.

The following Proposition describes how the local constant boundary data ${\boldsymbol\varphi}$ affect the free constants $C^\alpha$ defined in \eqref{introC}, and provides the rates for their closeness.

\begin{prop}\label{propu18}
Let $C^\alpha$ be given in \eqref{introC}. The following assertions hold.

(i) If  ${\boldsymbol\varphi}\in{\bf\Phi}_{1}$, then
\begin{align}\label{lamda0}
|C^1-1|\leq C\sqrt{\varepsilon},\quad 
|C^\alpha|\leq
\begin{cases}
C\varepsilon^{3/2},&\quad \alpha=2,\\
C\sqrt\varepsilon,&\quad \alpha=3.
\end{cases}
\end{align}

(ii) If  ${\boldsymbol\varphi}\in{\bf\Phi}_{2}$, then
\begin{align}\label{lamda20}
|C^\alpha|\leq 
C\sqrt{\varepsilon},\quad\alpha=1,3,\quad |C^2-1|\leq C\varepsilon.
\end{align}
\end{prop}

\begin{remark}This is a new observation, compared with the corresponding boundary estimates obtained before in \cite{BJL} for the Lam\'e system. Besides, taking case $(i)$ for instance, the closeness $|C^1-1|\leq C\sqrt{\varepsilon}$ can make the singularities caused by $\nabla{\bf u}_{1}$ and $\nabla{\bf u}_{0}$ be cancelled, see \eqref{v11v01} below.
This is also one of the difference from the interior case \cite{LX}, where the differences $|C_{1}^\alpha-C_{2}^\alpha|$,  $\alpha=1,2,3$, are in a small order for any boundary data. 
\end{remark}

\begin{remark}
Actually, for case $(i)$, if ${\boldsymbol\varphi}=a(1,0)^{\mathrm T}$ on $\Gamma_{2R}$, for any constant $a\in\mathbb{R}$, and ${\boldsymbol\varphi}\in{\bf\Phi}$, then we have
$$|C^1-a|\leq C\sqrt{\varepsilon},$$
to replace the analogous estimate in \eqref{lamda0}. A similar conclusion holds true also for case $(ii)$.
\end{remark}

Denote 
$$\delta(x_1)=\varepsilon+h_{1}(x_1)-h(x_1).$$ 
Then our first main result in 2D is as follows: 

\begin{theorem}\label{mainthm0}(Pointwise upper bounds)
Assume that $D,D_1,\Omega$, and $\varepsilon$ are defined as in Subsection \ref{Subsec1.2}, and ${\boldsymbol\varphi}\in{\bf\Phi}_{1}$ or  ${\boldsymbol\varphi}\in{\bf\Phi}_{2}$. Let ${\bf u}\in H^1(D;\mathbb R^2)\cap C^1(\bar{\Omega};\mathbb R^2)$ and $p\in L^2(D)\cap C^0(\bar{\Omega})$ be the solution to \eqref{Stokessys} . Then for sufficiently small $0<\varepsilon<1$,  we have
\begin{align*}%\label{gradDu}
|\nabla{{\bf u}}(x)|\leq
\frac{C}{\sqrt{\delta(x_1)}},\quad x\in \Omega_{R};\quad
\inf_{c\in\mathbb{R}}\|p+c\|_{C^0(\bar{\Omega}_{R})}\leq
\frac{C}{\varepsilon}\nonumber,
\end{align*}
and 
$$\|\nabla{\bf u}\|_{L^\infty(\Omega\setminus\Omega_R)}+\|p\|_{L^\infty(\Omega\setminus\Omega_R)}\leq C\|\boldsymbol\varphi\|_{C^{1,\alpha}(\partial D)},$$
where $C$ is a {\em{universal} constant}.
In particular,we have
\begin{equation*}
\|\nabla{{\bf u}}\|_{L^{\infty}(\Omega)}\leq 
\frac{C}{\sqrt\varepsilon}.
\end{equation*}
\end{theorem}
These upper bounds in this case are the same as the analogue of the interior estimates in \cite{LX}.

\subsection{Main Results for Locally Polynomial Boundary Data Case}

In this subsection we shall consider the following two kinds of polynomial boundary data:

\begin{equation}\label{defphi3}
{\bf\Phi}_{3}:=\left\{{\boldsymbol\varphi}\in{\bf\Phi}\big|~{\boldsymbol\varphi}=(x_1^{l_1},0)^{\mathrm T}~\mbox{on}~\Gamma_{2R}\right\},~{\bf\Phi}_{4}:=\left\{{\boldsymbol\varphi}\in{\bf\Phi}\big|~{\boldsymbol\varphi}=(0,x_1^{l_2})^{\mathrm T}~\mbox{on}~\Gamma_{2R}\right\},
\end{equation}
where $l_1,l_2\in\mathbb N^+$, and ${\bf\Phi}$ is defined in \eqref{defPhi}. Our method can be applied to more kinds of boundary data, for example, the polynomials of $x_{2}$. We left it to the interested readers. Set
\begin{equation}\label{narrowdelta}
\Omega_{\delta}(x_1):=\left\{(y_1,y_{2})\in\mathbb R^2\big| h(y_1)<y_{2}
<\varepsilon+h_{1}(y_1),\,|y_1-x_1|<\delta \right\},
\end{equation}
and define $$(q_{0})_{\Omega_R}=\frac{1}{|\Omega_R|}\int_{\Omega_R}q_{0}(y)\mathrm{d}y,$$
where $q_0=p_0-\bar p_0$ and $\bar p_0$ is the auxiliary function corresponding to $p_0$ constructed later, see \eqref{barp062} below for instance. For these general boundary values ${\boldsymbol\varphi}\in{\bf\Phi}_{3}$ and ${\boldsymbol\varphi}\in{\bf\Phi}_{4}$,  we first have the following general result to emphasize the dependence of the estimates on ${\boldsymbol\varphi}$, where $Q_{\beta}[\boldsymbol\varphi]$ are linear functionals of ${\boldsymbol\varphi}$ and ${\bf u}_0$ directly depend on ${\boldsymbol\varphi}$.

\begin{theorem}\label{mainthm}(Upper Bounds)
Assume that $D,D_1,\Omega$, and $\varepsilon$ are defined as in Subsection \ref{Subsec1.2}, and $\boldsymbol\varphi\in {\bf\Phi}$. Let ${\bf u}\in H^1(D;\mathbb R^2)\cap C^1(\bar{\Omega};\mathbb R^2)$ and $p\in L^2(D)\cap C^0(\bar{\Omega})$ be the solution to \eqref{Stokessys}. Then for sufficiently small $0<\varepsilon<1$, we have, $x\in \Omega_{R}$,
\begin{align*}
|\nabla{\bf u}(x)|&\leq \frac{C\sqrt\varepsilon}{\delta(x_1)}\Big(|Q_{1}[\boldsymbol\varphi]|+|Q_3[\boldsymbol\varphi]|+|x_1||Q_{2}[\boldsymbol\varphi]|\Big)+|\nabla{\bf u}_0(x)|,\\
\inf_{c\in\mathbb{R}}\|p+c\|_{C^0(\bar{\Omega}_{R})}&\leq \frac{C}{\varepsilon}\Big(|Q_{1}[\boldsymbol\varphi]|+|Q_3[\boldsymbol\varphi]|\Big)+\frac{C|Q_{2}[\boldsymbol\varphi]|}{\sqrt\varepsilon}+|p_0(x)-(q_0)_{\Omega_R}|,
\end{align*}
and 
\begin{align}\label{outside}
\|\nabla{\bf u}\|_{L^\infty(\Omega\setminus\Omega_R)}+\|p\|_{L^\infty(\Omega\setminus\Omega_R)}\leq C\|\boldsymbol\varphi\|_{C^{1,\alpha}(\partial D)},
\end{align}
where $C$ is a {\em{universal} constant},  $Q_{\beta}[\boldsymbol\varphi]$ is  defined in \eqref{aijbj}, $\beta=1,2,3$, and $({\bf u}_{0}, p_0)$ satisfies \eqref{equ_v3}.
\end{theorem}

The proof of \eqref{outside} is standard. Therefore, from Theorem \ref{mainthm}, to show the explicit effect from the above two kinds of $\boldsymbol\varphi$ is reduced to estimating $|Q_\beta[\boldsymbol\varphi]|$, $\beta=1,2,3$, and $|\nabla{\bf u}_0|$. Precisely, we have

\begin{prop}\label{lemaQb}
(1)	If ${\boldsymbol{\varphi}}\in{\bf\Phi}_{3}$, then	
	\begin{align*}
	|Q_1 [\boldsymbol\varphi]|, |Q_3[\boldsymbol\varphi]|\leq
	C,\quad
	|Q_2[\boldsymbol\varphi]|\leq
	\begin{cases}
	C|\ln\varepsilon|,&\mbox{for}~l_1=1,\\
	C,&\mbox{for}~l_1\geq 2.
	\end{cases}
	\end{align*}
	
(2)	If ${\boldsymbol{\varphi}}\in{\bf\Phi}_{4}$, then	
	\begin{align*}%\label{QB1}
	|Q_1[\boldsymbol\varphi]|, |Q_3[\boldsymbol\varphi]|\leq
	\begin{cases}
	\frac{C}{\sqrt{\varepsilon}},&\mbox{for}~l_2= 1,\\
	C,&\mbox{for}~l_2\geq 2,\\
	\end{cases}\quad
	|Q_2[\boldsymbol\varphi]|\leq
	\begin{cases}
	\frac{C}{\sqrt{\varepsilon}},&\mbox{for}~l_2= 1,\\
	C|\ln\varepsilon|,&\mbox{for}~l_2= 2,\\
	C,&\mbox{for}~ l_2\geq 3.
	\end{cases}
	\end{align*}
\end{prop}

\begin{prop}\label{propv0}
Let ${\bf u}_0\in{C}^{2}(\Omega;\mathbb R^2),~p_0\in{C}^{1}(\Omega)$ be the solution to \eqref{equ_v3}. Then  in $\Omega_{R}$, 
(1)	if ${\boldsymbol{\varphi}}\in{\bf\Phi}_{3}$, then 
	\begin{align*}
	|\nabla {\bf u}_0(x)|\leq\begin{cases}
	\frac{C}{\sqrt{\delta(x_1)}},&\mbox{for}~ l_1=1,\\
	C,&\mbox{for}~ l_1\geq 2,\\
	\end{cases}
	\end{align*}
and
	\begin{align*}
		\|p_0-(q_0^3)_{\Omega_R}\|_{L^{\infty}(\Omega_{\delta/2}(x_{1}))}\leq\begin{cases}
		\frac{C}{\varepsilon^{3/2}},&\mbox{for}~ l_1=1,\\
			\frac{C}{\varepsilon},&\mbox{for}~ l_1\geq 2.
			\end{cases}
	\end{align*}
(2)	If ${\boldsymbol{\varphi}}\in{\bf\Phi}_{4}$, then 
\begin{align*}
	|\nabla {\bf u}_0(x)|\leq
	\begin{cases}
		\frac{C}{\delta(x_1)},&\mbox{for}~ l_2=1,\\
		\frac{C}{\sqrt{\delta(x_1)}},&\mbox{for}~ l_2=2,\\
		C,& \mbox{for}~l_2\geq 3,\\
	\end{cases}
\end{align*}
and
\begin{align*}
\|p_0-(q_0^4)_{\Omega_R}\|_{L^{\infty}(\Omega_{\delta/2}(x_{1}))}\leq
\begin{cases}
	\frac{C}{\varepsilon^{3/2}},\quad&\mbox{for}~l_2=1,2,\\
	\frac{C}{\varepsilon}, &\mbox{for}~l_2\geq 3.
\end{cases}
\end{align*}
\end{prop}

Combining Theorem \ref{mainthm}, Proposition \ref{lemaQb} and Proposition \ref{propv0}, we immediately have 

\begin{theorem}\label{mainthm1}
Assume that $D,D_1,\Omega$, and $\varepsilon$ are defined as in Subsection \ref{Subsec1.2}, and ${\boldsymbol\varphi}\in{\bf\Phi}_{3}$. Let ${\bf u}\in H^1(D;\mathbb R^2)\cap C^1(\bar{\Omega};\mathbb R^2)$ and $p\in L^2(D)\cap C^0(\bar{\Omega})$ be the solution to \eqref{Stokessys}. Then for sufficiently small $0<\varepsilon<1$,  
	\begin{align*}%\label{gradDu}
	&|\nabla{{\bf u}}(x)|\leq
	\frac{C}{\sqrt{\delta(x_1)}},\quad x\in \Omega_{R};
	\quad \inf_{c\in\mathbb{R}}\|p+c\|_{C^0(\bar{\Omega}_{R})}\leq\begin{cases}
		\frac{C}{\varepsilon},& l_1=1,\\
		\frac{C}{\varepsilon^{3/2}},& l_1\geq 2,
	\end{cases}
	\end{align*}
where $C$ is a universal constant.
	In particular,
	\begin{equation*}
	\|\nabla{{\bf u}}\|_{L^{\infty}(\Omega)}\leq 
	\frac{C}{\sqrt\varepsilon}.
	\end{equation*}
\end{theorem}

\begin{theorem}\label{mainthm2}
Assume that $D,D_1,\Omega$, and $\varepsilon$ are defined as in Subsection \ref{Subsec1.2}, and ${\boldsymbol\varphi}\in{\bf\Phi}_{4}$. Let ${\bf u}\in H^1(D;\mathbb R^2)\cap C^1(\bar{\Omega};\mathbb R^2)$ and $p\in L^2(D)\cap C^0(\bar{\Omega})$ be the solution to \eqref{Stokessys}. Then for sufficiently small $0<\varepsilon<1$, 
	\begin{align*}
	&|\nabla{{\bf u}}(x)|\leq
	\begin{cases}
	\frac{C}{\delta(x_1)},& l_2=1,\\
	\frac{C}{\sqrt{\delta(x_1)}},& l_2\geq 2,\\
	\end{cases}\quad x\in \Omega_{R},
	\end{align*}
and	
	\begin{align*}\inf_{c\in\mathbb{R}}\|p+c\|_{C^0(\bar{\Omega}_{R})}\leq
\begin{cases}
	\frac{C}{\varepsilon^{3/2}},\quad&l_2=1,2,\\
	\frac{C}{\varepsilon}, &l_2\geq 3,
\end{cases}
	\end{align*}
where $C$ is a universal constant.
	In particular,
	\begin{equation*}
	\|\nabla{{\bf u}}\|_{L^{\infty}(\Omega)}\leq 
	\begin{cases}
	\frac{C}{\varepsilon},&\quad l_2=1,\\
	\frac{C}{\sqrt\varepsilon},&\quad l_2\geq 2.
	\end{cases}
	\end{equation*}
\end{theorem}

\subsection{Lower Bounds at the Midpoint of the Shortest Line}
From the above pointwise upper bounds obtained in Theorem \ref{mainthm0}, Theorem \ref{mainthm1} and Theorem \ref{mainthm2}, we find that the maximum of $|\nabla{\bf u}|$ may attain at the shortest line $\{|x_1|=0\}\cap\Omega$. In this subsection, we try to establish the lower bounds of $|\nabla{\bf u}|$ at the midpoint of the shortest line to show the optimality of these blow-up rate. To this aim, we need to consider the limiting case when $D_1$ touches $\partial D$ at the origin. Assume that $({\bf u}_\beta^*,p_\beta^*)$, $\beta=1,2,3$, verify
\begin{align}\label{defu*}
\begin{cases}
\nabla\cdot\sigma[{\bf u}_\beta^*,p_\beta^*]=0,\quad\nabla\cdot {\bf u}_\beta^*=0,&\mathrm{in}~\Omega^{0},\\
{\bf u}_\beta^*={\boldsymbol\psi}_{\beta},&\mathrm{on}~\partial{D}_{1}^{0}\setminus\{0\},\\
{\bf u}_\beta^*=0,&\mathrm{on}~\partial{D},
\end{cases}
\end{align} 
and $({\bf u}_0^*,p_0^*)$ is the unique solution of
\begin{align}\label{maineqn touch}
\begin{cases}
\nabla\cdot\sigma[{\bf u}_0^{*},p_0^{*}]=0,\quad\nabla\cdot {\bf u}_0^{ *}=0,\quad&\hbox{in}\ \Omega^{0},\\
{\bf u}_0^{*}=0,&\hbox{on}\ \partial D_{1}^{0},\\
{\bf u}_0^{*}={\boldsymbol\varphi},&\hbox{on}\ \partial{D}.
\end{cases}
\end{align}
We find that for different kinds of boundary data, the blow-up factor is different. 

First, for the locally constant boundary data case ${\boldsymbol\varphi}\in{\bf\Phi}_{1}$, one can see from Proposition \ref{lemaQb1} that the functionals $Q_{\beta}[{\boldsymbol\varphi}]$ defined in \eqref{defaij} tend to infinity as $\varepsilon\rightarrow0$, $\beta=1,3$. Because of this, we can not use their limits to be blow-up factors as in \cite{BJL} for Lam\'e system. To overcome this difficulty, we define a new linear bounded functional of ${\boldsymbol\varphi}$,
\begin{align*}
Q^*_{1,\beta}[{\boldsymbol\varphi}]=-\int_{\Omega^0} \left(2\mu e({\bf u}^*_{0}+{\bf u}^*_{1}),e({\bf u}^*_{\beta})\right)\mathrm{d}x,\quad\beta=1,2,3,
\end{align*}
which is the limit of $Q_{1,\beta}[{\boldsymbol\varphi}]:=Q_{\beta}[{\boldsymbol\varphi}]-a_{\beta 1},$ 
see Subsection \ref{subphi1}. 

Similarly, for ${\boldsymbol\varphi}\in{\bf\Phi}_{2}$, we define 
\begin{align*}
Q^*_{2,\beta}[{\boldsymbol\varphi}]=-\int_{\Omega^0} \left(2\mu e({\bf u}^*_{0}+{\bf u}^*_{2}),e({\bf u}^*_{\beta})\right)\mathrm{d}x,\quad\beta=1,2,3.
\end{align*}
These functionals $Q^*_{j,\beta}[{\boldsymbol\varphi}]$ are independent of $\varepsilon$, $j=1,2$. We will use them to define {\it blow-up factors}, which determines whether the blow-up occurs or not. 

Assume that the boundary data ${\boldsymbol\varphi}\in  C^{1,\alpha}(\partial D;\mathbb R^2)$ satisfies the following condition:
\begin{align*}
({\mathscr{H}}_1): ~Q_{1,1}^*[{\boldsymbol\varphi}]-(\kappa_1+\kappa) Q_{1,3}^*[{\boldsymbol\varphi}]\neq 0~\mbox{for~some}~{\boldsymbol\varphi}\in{\bf\Phi}_{1},\\
({\mathscr{H}}_2):~ Q_{2,1}^*[{\boldsymbol\varphi}]-(\kappa_1+\kappa)	 Q_{2,3}^*[{\boldsymbol\varphi}]\neq 0~\mbox{for~some}~{\boldsymbol\varphi}\in{\bf\Phi}_{2}.
\end{align*}
Then

\begin{theorem}\label{mainthm02}(Lower bounds)
Assume that $D,D_1,\Omega$, and $\varepsilon$ are defined as in Subsection \ref{Subsec1.2}, and ${\boldsymbol\varphi}\in{\bf\Phi}_{1}$ or  ${\boldsymbol\varphi}\in{\bf\Phi}_{2}$. Let ${\bf u}\in H^1(D;\mathbb R^2)\cap C^1(\bar{\Omega};\mathbb R^2)$ and $p\in L^2(D)\cap C^0(\bar{\Omega})$ be the solution to \eqref{Stokessys}. Then for sufficiently small $0<\varepsilon<1$, under assumption $({\mathscr{H}}_1)$ or $({\mathscr{H}}_2)$, the following lower bound estimate holds
\begin{align*}
|\nabla{\bf u}(0,\varepsilon/2)|\geq\frac{1}{C\sqrt{\varepsilon}}.
\end{align*} 
\end{theorem}
	
While, for the polynomial boundary data case ${\boldsymbol\varphi}\in{\bf\Phi}_{i}$, $i=3,4$, Proposition \ref{lemaQb} implies that all of the functionls $Q_{\beta}[{\boldsymbol\varphi}]$, $\beta=1,2,3$, are bounded when $l_1\geq2$ and $l_2\geq3$. This allows us to prove that as $\varepsilon\rightarrow0$, $Q_{\beta}[{\boldsymbol\varphi}]$ is convergent to the functional 
\begin{align*}
Q^*_{\beta}[{\boldsymbol\varphi}]:=\int_{\partial D_{1}^{0}}{\boldsymbol\psi}_\beta\cdot\sigma[{\bf u}_0^{*},p_0^{ *}]\nu,\quad\beta=1,2,3,
\end{align*}
and use them to define the blow-up factor as before, see Lemma \ref{zydyl} below. Assume that the boundary data ${\boldsymbol\varphi}\in  C^{1,\alpha}(\partial D;\mathbb R^2)$ satisfies the following condition:
	\begin{align*}
	({\mathscr{H}_3}): ~ Q_{1}^*[{\boldsymbol\varphi}]-(\kappa_1+\kappa)Q_{3}^*[{\boldsymbol\varphi}]\neq 0~\mbox{for~some}~{\boldsymbol\varphi}\in{\bf\Phi}_{i}, i=3,4.
	\end{align*}
Thus, for $l_1\geq 2$, and $l_2\geq 3$, we have the following lower bound of $|\nabla{\bf u}|$:
	
\begin{theorem}\label{mainthmlower}(Lower Bounds)
Assume that $D_1,D,\Omega$, and $\varepsilon$ are defined as above, and ${\boldsymbol\varphi}\in C^{1,\alpha}(\partial D;\mathbb R^2)$ for some $0<\alpha<1$. Let ${\bf u}\in H^1(D;\mathbb R^2)\cap C^1(\bar{\Omega};\mathbb R^2)$ and $p\in L^2(D)\cap C^0(\bar{\Omega})$ be a solution to \eqref{Stokessys}. Then there exists  a sufficiently small constant $\varepsilon_0>0$, such that for $0<\varepsilon<\varepsilon_0$, under the assumption $({\mathscr{H}_3})$, we have 
\begin{align*}
|\nabla{\bf u}(0,\varepsilon/2)|\geq\frac{1}{C\sqrt{\varepsilon}}.
\end{align*} 
\end{theorem}
	
\begin{remark}
Theorem \ref{mainthm02} and Theorem \ref{mainthmlower}, together with Theorem \ref{mainthm0}, Theorem \ref{mainthm1} and Theorem\ref{mainthm2}, imply that if ${\boldsymbol\varphi}$ satisfies one of the three conditions $({\mathscr{H}}_1)$, $({\mathscr{H}}_2)$ or $({\mathscr{H}_3})$, then the blow-up rate $\varepsilon^{-1/2}$ is optimal  in dimension two. 
\end{remark}

The rest of this paper is organized as follows. As mentioned before, from the decomposition of ${\bf u}$ and $p$ in \eqref{udecom}, we reduce the estimates of $\nabla{\bf u}$ and $p$ to the estimates of $\nabla{\bf u}_\alpha,\nabla{\bf u}_0$, $p_\alpha,p_0$, and the estimates of $C^\alpha$, $\alpha=1,2,3$.  For reader's convenience, we give some basic results about Stokes system in Section \ref{pre}. In Section \ref{estualpha}, we first construct the divergence free auxiliary functions ${\bf v}_\alpha$ and the corresponding associators $\bar p_\alpha$, $\alpha=1,2,3$, calculate their explicit pointwise estimates of the first and second-order derivatives, and then apply the energy iteration technique to prove that $\nabla{\bf v}_\alpha$ are major singular terms of $\nabla{\bf u}_\alpha$ and to obtain the estimates of $|\nabla \bf u_\alpha|$ and $p_\alpha$, respectively in three subsections.  Based on these estimates for $|\nabla \bf u_\alpha|$ and $p_\alpha$, in Section \ref{estCalpha} we establish the asymptotics and estimates of $a_{\alpha\beta}$, and show the matrix $(a_{\alpha\beta})_{3\times3}$ is invertible. In order to solve $C^{\alpha}$ and investigate the effect from the boundary data ${\boldsymbol\varphi}$, we divide into two cases. In Section \ref{estu0}, we consider the locally constant ${\boldsymbol\varphi}$, that is, ${\boldsymbol\varphi}\in{\bf\Phi}_i$, $i=1,2$, establish the estimates of $(\nabla{\bf u}_0,p_{0})$ and $Q_\beta[{\boldsymbol\varphi}]$, and then complete the proof of Theorem \ref{mainthm0}. In Section \ref{sec5}, we consider the locally polynomial ${\boldsymbol\varphi}$, that is, ${\boldsymbol\varphi}\in{\bf\Phi}_i$, $i=3,4$, and prove Theorem \ref{mainthm1} and Theorem \ref{mainthm2}. Finally, in Section \ref{sec6} we give the lower bounds to show the optimality of the blow-up rates obtained above and prove Theorem \ref{mainthm02} and Theorem \ref{mainthmlower}.

\section{Preliminaries}\label{pre}
In this section, we give some basic results about Stokes system. The first one is the  well-known maximum modulus estimate in a Lipschitz domain. We refer to the book of Ladyzhenskaya \cite{Lady} about the proof.

\begin{lemma}\label{lemmaxi}
Let $\Omega\subset\mathbb R^2$ be a bounded Lipschitz domain. Suppose that $({\bf u},p)$ is the solution of the Stokes problem
\begin{align*}
\begin{cases}
\mu\Delta{\bf u}=\nabla p,\quad\nabla\cdot{\bf u}=0,&\quad\mbox{in}~\Omega,\\
{\bf u}={\bf g},&\quad\mbox{on}~\partial\Omega,
\end{cases}
\end{align*}
where ${\bf g}\in C^0(\partial\Omega)$ is given. Then 
\begin{equation*}
\|{\bf u}\|_{L^\infty(\Omega)}\leq C\|{\bf u}\|_{L^\infty(\partial\Omega)},
\end{equation*}
where $C>0$ is a constant depending on $\Omega$. 
\end{lemma}

Next we  give the following $L^{p}$ estimate for Stokes flow in a bounded domain, with partially vanishing boundary data, which can be found in \cite[Theorem IV.5.1]{GaldiBook}.

\begin{theorem}\label{thmWmq}
	Let $\Omega$ be an arbitrary domain in $\mathbb R^d$, $d\geq2$, with a boundary portion $\sigma$ of class $C^{m+2}$, $m\geq0$. Let $\Omega_0$ be any bounded subdomain of $\Omega$ with $\partial\Omega_0\cap\partial\Omega=\sigma$. Further, let
	\begin{align*}
	{\bf u}\in W^{1,q}(\Omega_0), \quad p\in L^q(\Omega_0),\quad 1<q<\infty,
	\end{align*}
	be such that
	\begin{align*}
	(\nabla{\bf u},\nabla{\boldsymbol\psi})&=-\langle{\bf f},{\boldsymbol\psi}\rangle+(p,\nabla\cdot{\boldsymbol\psi}),\quad\mbox{for~all}~{\boldsymbol\psi}\in C_0^\infty(\Omega_0),\\
	({\bf u},\nabla\varphi)&=0,\quad \mbox{for~all}~\varphi\in C_0^\infty(\Omega_0),\\
	{\bf u}&=0,\quad \mbox{at}~\sigma.
	\end{align*}
	Then, if ${\bf f}\in W^{m,q}(\Omega_0)$, we have
	$${\bf u}\in W^{m+2,q}(\Omega'),\quad p\in W^{m+1,q}(\Omega'),$$
	for any $\Omega'$ satisfying 
	\begin{enumerate}
		\item $\Omega'\subset\Omega$,
		\item $\partial\Omega'\cap\partial\Omega$ is a strictly interior subregion of $\sigma$.
	\end{enumerate}
	Finally, the following estimate holds
	\begin{align*}
	\|{\bf u}\|_{W^{m+2,q}(\Omega')}+\|p\|_{W^{m+1,q}(\Omega')}\leq C\left(\|{\bf f}\|_{W^{m,q}(\Omega_0)}+\|{\bf u}\|_{W^{1,q}(\Omega_0)}+\|p\|_{L^{q}(\Omega_0)}\right),
	\end{align*}
	where $C=C(d,m,q,\Omega',\Omega_0)$.
\end{theorem}

Consider the following general boundary value problem:
\begin{align}\label{w1}
\begin{cases}
-\mu\,\Delta{\bf w}+\nabla q={\bf f}:=\mu\,\Delta{\bf v}-\nabla\bar{p},\quad&\mathrm{in}\,\Omega,\\
\nabla\cdot {\bf w}=0,\quad&\mathrm{in}\,\Omega_{2R},\\
\nabla\cdot {\bf w}=-\nabla\cdot {\bf v},\quad&\mathrm{in}\,\Omega\setminus\Omega_R,\\
{\bf w}=0,\quad&\mathrm{on}\,\partial\Omega,
\end{cases}
\end{align}
where $\Omega_R$ is defined in \eqref{narrow-region}. Then by applying Theorem \ref{thmWmq} and using the bootstrap argument, we have the following local $W^{1,\infty}$ and $W^{2,\infty}$ estimates for Stokes system.
\begin{prop}\label{lemWG2}{\cite[Proposition 3.6]{LX}}
	Let $({\bf w},q)$ be the solution to \eqref{w1}. Then the following estimate holds
	\begin{align*}%\label{W2pstokes}
	\|\nabla {\bf w}\|_{L^{\infty}(\Omega_{\delta/2}(z_1))}&\leq
	\,C\left(\delta^{-1}\|\nabla{\bf w}\|_{L^{2}(\Omega_{\delta}(z_1))}+\delta\|{\bf f}\|_{L^{\infty}(\Omega_{\delta}(z_1))} \right),\\
		\|\nabla {\bf w}\|_{L^{\infty}(\Omega_{\delta/2}(z_1))}&+\|\nabla q\|_{L^{\infty}(\Omega_{\delta/2}(z_1))}\nonumber\\
	&\leq
	\,C\left(\delta^{-2}\|\nabla{\bf w}\|_{L^{2}(\Omega_{\delta}(z_1))}+\|{\bf f}\|_{L^{\infty}(\Omega_{\delta}(z_1))}+\delta\|\nabla{\bf f}\|_{L^{\infty}(\Omega_{\delta}(z_1))}\right),
	\end{align*}
	where $|z_1|<R$, $\delta=\delta(z_1)$, and $\Omega_{\delta}(z_1)$ is defined in \eqref{narrowdelta}.
\end{prop}
We shall use Proposition \ref{lemWG2} to establish the $L^\infty$ estimate of $|\nabla{\bf w}|$ in $\Omega_R$. See the proof of Proposition \ref{propu11} below for the details.

\section{Estimates of $|\nabla \bf u_\alpha|$ and $p_\alpha$, ~$\alpha=1,2,3$}\label{estualpha} 
In this section, we will estimate $|\nabla \bf u_\alpha|$ and $p_\alpha$, $\alpha=1,2,3$, respectively, in Subsection \ref{subsec2.1}, Subsection \ref{subsec2.2}, and Subsection \ref{subsec2.3}. 

In what follows, let us denote
\begin{align*}
\delta(x_1):=\epsilon+h_1(x_1)-h(x_1),\quad\mbox{for}\, |x_1|\leq 2R.
\end{align*}
First, we introduce the Keller-type auxiliary function $k(x)\in C^{2}(\mathbb{R}^{2})$,
satisfying $k(x)=\frac{1}{2}$ on $\partial D_{1}$, $k(x)=-\frac{1}{2}$ on $\partial D$, especially,
\begin{equation}\label{def_vx}
k(x)=\frac{x_{2}-h(x_1)}{\delta(x_1)}-\frac{1}{2}\quad\hbox{in}\ \Omega_{2R},
\end{equation}
and $\|k\|_{C^{2}(\mathbb{R}^{2}\backslash\Omega_{R})}\leq C$.
In order to express our idea clearly and avoid unnecessary difficulty from computation, without loss of generality, we assume that $h_{1}$ and $h$ are quadratic, say, $h_1(x_1)=\kappa_1x_1^2$, $h(x_1)=\kappa x_1^2$ for $x_1\leq 2R$, where $\kappa_1$ and $\kappa$ are two positive constants and $\kappa_0=\kappa_1-\kappa>0$. 
Thus,
\begin{equation}\label{55}
\delta(x_1)=\varepsilon+(\kappa_1-\kappa)x_1^2,\quad\mbox{for}\, |x_1|\leq 2R.
\end{equation}
A direct calculation gives, in $\Omega_{2R}$,
\begin{align*}%\label{estkx}
\quad\partial_{x_1}k(x)&= -\frac{2\kappa x_1}{\delta(x_1)}-2(\kappa_1-\kappa)\frac{x_1}{\delta(x_1)}\Big(k(x)-\frac{1}{2}\Big)\\&= -(\kappa_1+\kappa)\frac{x_1}{\delta(x_1)}-2(\kappa_1-\kappa)\frac{x_1}{\delta(x_1)}k(x),
\end{align*}
and
$$\partial_{x_2}k(x)=\frac{1}{\delta(x_1)}.\hspace{4.5cm}$$

Next we shall use such $k(x)$ to construct a series of vector-valued auxiliary functions ${\bf v}_{\alpha}\in C^{2}(\Omega;\mathbb R^2)$, satisfying the incompressible condition, and the corresponding associators $\bar p_{\alpha}\in C^{1}(\Omega)$. Then we  apply proposition 3.6 in \cite{LX} to establish the $L^{\infty}$ estimates of $|\nabla({\bf u}_\alpha-{\bf v}_\alpha)|$ and $|p_{\alpha}-\bar p_{\alpha}|$ in $\Omega_{R}$, and show that $\nabla{\bf v}_\alpha$ and $\bar p_{\alpha}$ capture the main singularities of $|\nabla{\bf u}_\alpha|$ and $p_\alpha$, $\alpha=1,2,3$. The needed basic tools and framework can refer to section 3 of \cite{LX}. 

\subsection{Estimates of $|\nabla {\bf u}_1|$ and $p_1$}\label{subsec2.1}

We construct a vector-valued ${\bf v}_{1}\in C^{2}(\Omega;\mathbb R^2)$, such that ${\bf v}_{1}={\bf u}_{1}={\boldsymbol\psi}_{1}$ on $\partial{D}_{1}$ and ${\bf v}_{1}={\bf u}_{1}=0$ on $\partial{D}$, and satisfying the divergence free condition in $\Omega_{2R}$, especially,
\begin{align}\label{v1alpha}
{\bf v}_{1}=\boldsymbol\psi_{1}\Big(k(x)+\frac{1}{2}\Big)+(\kappa_1+\kappa)\begin{pmatrix}
\frac{-4x_1^2}{\delta(x_{1})}+\frac{1}{\kappa_1-\kappa}\\\\
G_1(x)
\end{pmatrix}
\Big(k^2(x)-\frac{1}{4}\Big)\quad \mbox{in}~\Omega_{2R},
\end{align}
where
\begin{align}\label{defG1}
G_1(x)=\,\,2x_1\bigg(\Big(\kappa+(\kappa_1-\kappa)\Big(k(x)+\frac{1}{2}\Big)\Big)\frac{-4x_1^2}{\delta(x_{1})}+2\Big(k(x)+\frac{1}{2}\Big)\bigg),
\end{align}
and 
$\|{\bf v}_{1}\|_{C^{2}(\Omega\setminus\Omega_{R})}\leq\,C$. We remark that, due to the asymmetry of $h_{1}$ and $h$,  here the construction of ${\bf v}_{1}$ is more complicated than that  constructed in \cite{LX} (see (2.7) there).

Because it is easy to find by a direct calculations (see \eqref{v11-222-2D} and \eqref{v12-222-2D} below) that $|\Delta{\bf v}_{1}|\sim\frac{1}{\delta^2(x_1)}$ is too large to apply proposition 3.6 in \cite{LX} to obtain our desired results. To this end, we choose an associated function $\bar{p}_1\in C^{1}(\Omega)$, such that
\begin{equation}\label{defp113D}
\bar{p}_1=\frac{\kappa_1+\kappa}{\kappa_1-\kappa}\frac{2\mu\, x_{1}}{\delta^2(x_1)}+\mu\,\partial_{x_2} {\bf v}_1^{(2)},\quad\mbox{in}~\Omega_{2R},
\end{equation}
and  $\|\bar{p}_1\|_{C^{1}(\Omega\setminus\Omega_{R})}\leq C$, to make  $|\mu\,\Delta{\bf v}_{1}-\nabla\bar{p}_1|$ much smaller than $|\mu\,\Delta{\bf v}_{1}|$ in $\Omega_{2R}$, see  \eqref{estdivv11p1} below. We denote the differences
\begin{align*}%\label{w}
{\bf w}_{1}:={\bf u}_{1}-{\bf v}_{1},\quad q_{1}:=p_{1}-\bar{p}_{1},
\end{align*}
and thus $({\bf w}_{1},q_{1})$ satisfies the general boundary value problem \eqref{w1}. 

For $|x_1|,t\leq\,R$, denote
\begin{align}\label{omega_t}
\Omega_{t}(x_1):=\left\{(y_1,y_{2})\big|~h(y_1)<y_{2}
 <\varepsilon+h_1(y_1),~\,|y_1-x_1|<t \right\},
 \end{align}
and $$(q_{i})_{\Omega_R}=\frac{1}{|\Omega_R|}\int_{\Omega_R}q_{i}(y)\mathrm{d}y.$$
Then, the following Proposition shows that such $(\nabla{\bf v}_{1},\bar{p}_{1})$ constructed above captures the main singular terms of $(\nabla{\bf u}_{1},p_{1})$ in $\Omega_{2R}$.
\begin{prop}\label{propu11}
	Let ${\bf u}_{1}\in{C}^{2}(\Omega;\mathbb R^2),~p_{1}\in{C}^{1}(\Omega)$ be the solution to \eqref{equ_v1}. Then we have
	\begin{equation*}
	\|\nabla({\bf u}_{1}-{\bf v}_{1})\|_{L^{\infty}(\Omega_{\delta/2}(x_1))}\leq C,~~ \,x\in\Omega_{R},
	\end{equation*}
and 
   \begin{equation*}
	\|\nabla^2({\bf u}_{1}-{\bf v}_{1})\|_{L^{\infty}(\Omega_{\delta/2}(x_1))}+\|\nabla q_1\|_{L^{\infty}(\Omega_{\delta/2}(x_1))}\leq \frac{C}{\delta(x_1)}.
    \end{equation*}
	Consequently, 
	\begin{align}\label{upperu1}
	\frac{1}{C\delta(x_1)}\leq|\nabla {\bf u}_{1}(x)|\leq \frac{C}{\delta(x_1)},\quad\,x\in\Omega_{R},
	\end{align}
	and
	\begin{align}\label{upperp1}
	\|p_{1}-(q_{1})_{\Omega_R}\|_{L^{\infty}(\Omega_{\delta/2}(x_{1}))}\leq\,\frac{C}{\varepsilon^{3/2}},\quad\,x\in\Omega_{R}.
	\end{align}
\end{prop}

We will employ the procedure built in section 3 of \cite{LX} to prove Proposition \ref{propu11}. First, by applying the energy estimates together with the estimates of $|{\bf f}|$, we show the boundedness of the global energy of ${\bf w}_1$, see Lemma \ref{lem3.0}. Then, by making use of the iteration technique, we obtain estimate of the local energy in $\Omega_\delta(z_1)$, see Lemma \ref{lem3.1}. Finally, by using a scaling argument and combining with $L^p$ estimate for Stokes flow in a bounded domain with partially vanishing boundary data, we prove Proposition \ref{propu11}, and establish the pointwise estimates \eqref{upperu1} and \eqref{upperp1}. 

First, by a direct calculation, for $x\in\Omega_{2R}$,
\begin{align}\label{estv112}
\partial_{x_1}{\bf v}_1^{(1)}&=\partial_{x_1}k(x)-8(\kappa_1+\kappa)x_1\Big(\frac{1}{\delta(x_1)}+\frac{(\kappa_1-\kappa)x_1^2}{\delta^2(x_1)}\Big)\Big(k^2(x)-\frac{1}{4}\Big)\nonumber\\
&\quad+2(\kappa_1+\kappa)k(x)\partial_{x_1}k(x)\Big(\frac{-4x_1^2}{\delta(x_{1})}+\frac{1}{\kappa_1-\kappa}\Big), \\
\partial_{x_2}{\bf v}_1^{(1)}&=\frac{1}{\delta(x_{1})}+\frac{2(\kappa_1+\kappa)k(x)}{\delta(x_{1})}\Big(\frac{-4x_1^2}{\delta(x_{1})}+\frac{1}{\kappa_1-\kappa}\Big),\label{estv1112}\\
|\partial_{x_1}{\bf v}_1^{(2)}|&\leq C,\label{estv1122}
\end{align}
and
\begin{align}\label{estv11223}
\partial_{x_2}{\bf v}_1^{(2)}&=\frac{2(\kappa_1+\kappa) }{\delta(x_1)}\bigg(x_{1}\Big(\frac{-4 (\kappa_1-\kappa)x_{1}^2}{\delta(x_1)}+2\Big)\Big(k^2(x)-\frac{1}{4}\Big)+k(x)G_1(x)\bigg).
\end{align}
It is obvious from these estimates that	
\begin{equation*}%\label{divfree1}
\nabla\cdot{\bf v}_{1}=0,\quad\mbox{in}~\Omega_{2R},
\end{equation*} 
moreover, 
\begin{align}\label{v1u1}
|\nabla{\bf v}_{1}(x)|\leq\,\frac{C}{\delta(x_1)},\quad\mbox{in}~\Omega_{2R}.
\end{align}

For the  second-order derivatives, a further calculation yields
\begin{align}
|\partial_{x_1x_1}{\bf v}_1^{(1)}|&\leq\frac{C}{\delta(x_{1})},\quad\partial_{x_2x_2}{\bf v}_1^{(1)}=\frac{2(\kappa_1+\kappa)}{\delta^2(x_{1})}\Big(\frac{-4x_1^2}{\delta(x_{1})}+\frac{1}{\kappa_1-\kappa}\Big),\label{v11-222-2D}\\
|\partial_{x_1x_1}{\bf v}_1^{(2)}|&\leq\frac{C|x_1|}{\delta(x_{1})},\quad|\partial_{x_1x_2}{\bf v}_1^{(2)}|\leq\frac{C}{\delta(x_{1})},\quad|\partial_{x_2x_2}{\bf v}_1^{(2)}|\leq\frac{C|x_1|}{\delta^2(x_{1})};\label{v12-222-2D}
\end{align}
and for the associated $\bar{p}_1$,
\begin{align}
\partial_{x_1}\bar{p}_1&=\frac{2(\kappa_1+\kappa)\mu}{\delta^2(x_{1})}\Big(\frac{-4x_1^2}{\delta(x_{1})}+\frac{1}{\kappa_1-\kappa}\Big)+\mu\,\partial_{x_{1}x_2} {\bf v}_1^{(2)}.\label{v11-222-2g}
\end{align}
Thus, combining  \eqref{v11-222-2D} and \eqref{v11-222-2g}, we have 
$$\mu\,\partial_{x_2x_2}{\bf v}_1^{(1)}-\partial_{x_1}\bar{p}_1=-\mu\,\partial_{x_1x_2}{\bf v}_1^{(2)}.$$
While, it is easy to see from the definition of $\bar{p}_1$ in \eqref{defp113D} that
$$\mu\,\partial_{x_2x_2}{\bf v}_1^{(2)}-\partial_{x_2}\bar{p}_1=0.$$
Therefore, together with \eqref{v11-222-2D}--\eqref{v11-222-2g}, we deduce
\begin{align}\label{estdivv11p1}
|{\bf f}_{1}|=\left|\mu\,\Delta{\bf v}_{1}-\nabla\bar{p}_1\right|=\left|\begin{pmatrix}
\mu\,\partial_{x_1x_1}{\bf v}_1^{(1)}-\mu\,\partial_{x_1x_2}{\bf v}_1^{(2)}\\\\
\mu\,\partial_{x_1x_1}{\bf v}_1^{(2)}
\end{pmatrix}\right|\leq \frac{C}{\delta(x_{1})}.
\end{align}

We will use a similar argument as in the proof of \cite[Lemma 4.1]{LX} to prove the global boundedness of $|\nabla{\bf w}_1|$. Before this, let us recall the following lemma from \cite{LX}.

\begin{lemma}\label{lemmaenergy}{\cite[Lemma 3.7]{LX}}
	Let $({\bf w},q)$ be the solution to \eqref{w1}. Then if ${\bf v}\in C^{2}(\Omega;\mathbb R^2)$ and $\bar{p}\in C^{1}(\Omega)$ satisfy
	\begin{equation*}%\label{estv113D3}
	\|{\bf v}\|_{C^{2}(\Omega\setminus\Omega_R)}\leq\,C,\quad\|\bar{p}\|_{C^1(\Omega\setminus\Omega_{R})}\leq C,
	\end{equation*}
	and 
	\begin{align}\label{int-fw}
	\Big| \int_{\Omega_{R}}\sum_{j=1}^{2}{\bf f}^{(j)}{\bf w}^{(j)}\mathrm{d}x\Big|\leq\,C\left(\int_{\Omega}|\nabla {\bf w}|^2\mathrm{d}x\right)^{1/2},
	\end{align}
	then
	\begin{align}\label{estw11case2}
	\int_{\Omega}|\nabla {\bf w}|^{2}\mathrm{d}x\leq C.
	\end{align}
\end{lemma}

Now we are ready to prove the following result. 
\begin{lemma}\label{lem3.0}
	Let $({\bf w}_1,q_1)$ be the solution to \eqref{w1}. Then
	\begin{align}\label{w1alpha}
	\int_{\Omega}|\nabla{\bf w}_{1}|^{2}\mathrm{d}x\leq\,C.
	\end{align}
\end{lemma}
\begin{proof}
	By the Sobolev trace embedding theorem, we have 
	\begin{align}\label{w1Dw1}
	\int_{\substack{|x_1|=R,\\h(x_1)<x_{2}<\varepsilon+h_{1}(x_1)}}|{\bf w}_1|\mathrm{d}x_{2}
	&\leq C\left(\int_{\Omega}|\nabla {\bf w}_1|^2\mathrm{d}x\right)^{1/2}.
	\end{align}
	Recalling \eqref{estdivv11p1}, applying the  integration by parts with respect to $x_1$, and in view of \eqref{w1Dw1}, we obtain
	\begin{align*}%\label{w11p11}
	\left|\int_{\Omega_{R}} {\bf f}_1^{(1)}{\bf w}_1^{(1)}\mathrm{d}x\right|
	&=\left|\int_{\Omega_{R}} {\bf w}_1^{(1)}\big(\mu\,\partial_{x_1x_1}{\bf v}_1^{(1)}-\mu\partial_{x_1x_2}{\bf v}_1^{(2)}\big)\mathrm{d}x\right|\nonumber\\
	&\leq\,C\int_{\Omega_{R}}|\partial_{x_1}{\bf w}_1^{(1)}||\partial_{x_1}{\bf v}_1^{(1)}|+|\partial_{x_1}{\bf w}_1^{(1)}||\partial_{x_2}{\bf v}_1^{(2)}|\mathrm{d}x\nonumber\\
	&\quad+C\,\int_{\substack{|x_1|=R,\\h(x_1)<x_{2}<\varepsilon+h_{1}(x_1)}}|{\bf w}_1^{(1)}|\mathrm{d}x_{2}\\
	&\leq C\int_{\Omega_{R}}|\partial_{x_1}{\bf w}_1^{(1)}||\partial_{x_1}{\bf v}_1^{(1)}|+|\partial_{x_1}{\bf w}_1^{(1)}||\partial_{x_2}{\bf v}_1^{(2)}|\mathrm{d}x\nonumber\\
	&\quad+C\left(\int_{\Omega}|\nabla {\bf w}_1|^2\mathrm{d}x\right)^{1/2}=:\mathrm{I}_{11}+\mathrm{I}_{12}.
	\end{align*}
	By virtue of  \eqref{estv112} and \eqref{estv11223}, 
	\begin{equation*}%\label{estv113D}
	|\partial_{x_1}{\bf v}_1^{(1)}|,|\partial_{x_2}{\bf v}_1^{(2)}|\leq\frac{C|x_1|}{\delta(x_1)}\quad
	\mbox{ in} ~\Omega_{2R}.
	\end{equation*}
	This, together with H\"{o}lder's inequality, implies
	\begin{align*}
	|\mathrm{I}_{11}|&\leq C\left(\int_{\Omega_{R}}|\partial_{x_1}{\bf v}_1^{(1)}|^{2}+|\partial_{x_2}{\bf v}_1^{(2)}|^2\mathrm{d}x\right)^{1/2}
	\left(\int_{\Omega} |\partial_{x_1}{\bf w}_1^{(1)}|^2\mathrm{d}x\right)^{1/2}\nonumber\\
	&\leq C \left(\int_{\Omega} |\nabla {\bf w}_1|^2\mathrm{d}x\right)^{1/2}.
	\end{align*}
	Thus, we conclude that
	\begin{align*}%\label{w11p11}
	\left|\int_{\Omega_{R}} {\bf f}_1^{(1)}{\bf w}_1^{(1)}\mathrm{d}x\right|
	\leq C\left(\int_{\Omega}|\nabla {\bf w}_1|^2\mathrm{d}x\right)^{1/2}.
	\end{align*}
	
	Similarly, using \eqref{estdivv11p1}, \eqref{estv1122}, and \eqref{w1Dw1}, we derive
	\begin{align*}%\label{estw13vp13}
	&\left|\int_{\Omega_{R}} {\bf f}_1^{(2)}{\bf w}_1^{(2)}\mathrm{d}x\right|=\left|\mu\,\int_{\Omega_{R}} {\bf w}_1^{(2)}\partial_{x_1x_1}{\bf v}_1^{(2)}\mathrm{d}x\right|\nonumber\\
	&\leq\,C\int_{\Omega_{R}}|\partial_{x_1}{\bf w}_1^{(2)}||\partial_{x_1}{\bf v}_1^{(2)}|\mathrm{d}x+C\int_{\substack{|x'|=R,\\h(x_1)<x_{2}<\varepsilon+h_{1}(x')}}|{\bf w}_1^{(2)}|\mathrm{d}x_{2}\\
	&\leq  C \left(\int_{\Omega} |\nabla {\bf w}_1|^2\mathrm{d}x\right)^{1/2}.
	\end{align*}
	Therefore, 
	\begin{align*}%\label{int-fw}
	\Big| \int_{\Omega_{R}}\sum_{j=1}^{2}{\bf f}_1^{(j)}{\bf w}_1^{(j)}\mathrm{d}x\Big|\leq\,C\left(\int_{\Omega}|\nabla {\bf w}_1|^2\mathrm{d}x\right)^{1/2}.
	\end{align*}
	This is exactly the condition \eqref{int-fw}, and so the proof of Lemma \ref{lem3.0} is finished.
\end{proof}

Before proving the estimate of the local energy of $|\nabla{\bf w}_1|$, we first recall the  Caccioppoli-type inequality from \cite[Lemma 3.10]{LX}.
\begin{lemma} (Caccioppoli-type Inequality) Let $({\bf w},q)$ be the solution to \eqref{w1}. For $0<t<s\leq R$, there holds
	\begin{align}\label{iterating1}
	\int_{\Omega_{t}(z')}|\nabla {\bf w}|^{2}\mathrm{d}x\leq\,&
	\frac{C\delta^2(z')}{(s-t)^{2}}\int_{\Omega_{s}(z')}|\nabla {\bf w}|^2\mathrm{d}x+C\left((s-t)^{2}+\delta^{2}(z')\right)\int_{\Omega_{s}(z')}|{\bf f}|^{2}\mathrm{d}x,
	\end{align}
	where $\Omega_{t}(z_1)$ is defined by \eqref{omega_t}. 
\end{lemma}
By the iteration technique used in \cite{LLBY,BLL,LX}, we derive the following lemma. 	
\begin{lemma}\label{lem3.1}
	Let $({\bf w}_1,q_1)$ be the solution to \eqref{w1}. Then
	\begin{align}\label{estw11narrow}
	\int_{\Omega_{\delta}(z_1)}|\nabla {\bf w}_1|^{2}\mathrm{d}x\leq C\delta^2(z_1).
	\end{align}
\end{lemma}

\begin{proof} 
Denote
	$$F(t):=\int_{\Omega_{t}(z_1)}|\nabla {\bf w}_1|^{2}\mathrm{d}x.$$
	From 	\eqref{estdivv11p1}, we have 
	\begin{align*}%\label{f1alpha}
	\int_{\Omega_{s}(z_1)}|{\bf f}_{1}|^{2}\mathrm{d}x\leq\frac{Cs}{\delta(z_1)}.
	\end{align*}
	Then we obtain from \eqref{iterating1} that
	\begin{align}\label{iteration3D}
	F(t)\leq \left(\frac{c_{0}\delta(z_1)}{s-t}\right)^2 F(s)+C\left((s-t)^{2}+\delta^2(z_1\right)\frac{s}{\delta(z_1)},
	\end{align}
	where $c_{0}$ is a constant and we fix it now. 
	
Let $k_{0}=\left[\frac{1}{4c_{0}\sqrt{\delta(z_1)}}\right]+1$ and $t_{i}=\delta(z_1)+2c_{0}i\delta(z_1), i=0,1,2,\dots,k_{0}$. Applying \eqref{iteration3D} with $s=t_{i+1}$ and $t=t_{i}$ leads to the following iteration formula:
	\begin{align*}%\label{18}
	F(t_{i})\leq \frac{1}{4}F(t_{i+1})+C(i+1)\delta^2(z_1).
	\end{align*}
	After $k_{0}$ iterations, and using \eqref{w1alpha}, we obtain
	\begin{align*}
	F(t_0)&\leq \left(\frac{1}{4}\right)^{k}F(t_{k})
	+C\delta^2(z_1)\sum\limits_{l=0}^{k-1}\left(\frac{1}{4}\right)^{l}(l+1)\nonumber\\
	&\leq C\left(\frac{1}{4}\right)^{k}+C\delta^2(z_1)\sum\limits_{l=0}^{k-1}\left(\frac{1}{4}\right)^{l}(l+1)\leq C\delta^2(z_1),
	\end{align*}
	for sufficiently small $\varepsilon$ and $|z_1|$.  This gives \eqref{estw11narrow}, and thus the lemma is proved.
\end{proof}

Now we are ready to complete the proof of Proposition \ref{propu11}.
\begin{proof}[Proof of Proposition \ref{propu11}.]
Using Proposition \ref{lemWG2}, we have
	\begin{align*}%\label{W2pstokes}
	&\|\nabla {\bf w}_1\|_{L^{\infty}(\Omega_{\delta/2}(z_1))}
	\leq
	\,C\left(\delta^{-1}(z_1)\|\nabla {\bf w}_1\|_{L^{2}(\Omega_{\delta}(z_1))}+\delta(z_1)\|{\bf f}_1\|_{L^{\infty}(\Omega_{\delta}(z_1))} \right).
	\end{align*}
	Combining \eqref{estdivv11p1} and \eqref{estw11narrow} yields
	\begin{align*}%\label{W2pstokes}
	\|\nabla {\bf w}_1\|_{L^{\infty}(\Omega_{\delta/2}(z_1))}\leq C.
	\end{align*}
	Since ${\bf w}_1={\bf u}_1-{\bf v}_1$, it follows from \eqref{v1u1} that
	\begin{align*}
	|\nabla {\bf u}_1|&\leq|\nabla {\bf v}_1(z)|+|\nabla {\bf w}_1(z)|\leq \frac{C}{\delta(z_1)},\quad \quad~z\in\Omega_{R}.
	\end{align*}
Similarly, we obtain from Proposition \ref{lemWG2}, \eqref{estdivv11p1}, and \eqref{estw11narrow}, 
\begin{align*}
		\|\nabla^2 {\bf w}_1\|_{L^{\infty}(\Omega_{\delta/2}(z_1))}&+\|\nabla q_1\|_{L^{\infty}(\Omega_{\delta/2}(z_1))}\nonumber\\
	&\leq
	\,C\left(\delta^{-2}\|\nabla{\bf w}_1\|_{L^{2}(\Omega_{\delta}(z_1))}+\|{\bf f}_1\|_{L^{\infty}(\Omega_{\delta}(z_1))}+\delta\|\nabla{\bf f}_1\|_{L^{\infty}(\Omega_{\delta}(z_1))}\right)\\
	&\leq \frac{C}{\delta(z_1)}.
\end{align*}
Combining $q_1=p_1-\bar{p}_1$, the mean value theorem, and \eqref{defp113D}, we deduce
\begin{align*}
	|p_1(z)-(q_1)_{\Omega_R}|\leq |q_1-(q_1)_{\Omega_R}|+|\bar{p}_1|\leq \frac{C}{\varepsilon}+\frac{C|z_1|}{\delta^2(z_1)}\leq\frac{C}{\varepsilon^{3/2}},\quad z\in\Omega_{R}.
\end{align*}
	Therefore, Proposition \ref{propu11} is proved. 
\end{proof}

\subsection{Estimates of $|\nabla {\bf u}_2|$ and $p_2$}\label{subsec2.2}

By a similar way, we construct ${\bf v}_{2}\in C^{2}(\Omega;\mathbb R^2)$ satisfying
\begin{align}\label{v2alpha}
{\bf v}_{2}=\boldsymbol\psi_{2}\Big(k(x)+\frac{1}{2}\Big)+\begin{pmatrix}
\frac{6x_1}{\delta(x_{1})}\\\\
G_2(x)
\end{pmatrix}
\Big(k^2(x)-\frac{1}{4}\Big)\quad \mbox{in}~\Omega_{2R},
\end{align}
where
\begin{equation}\label{defG2}
G_2(x)=-2k(x)+\frac{6x_1}{\delta(x_1)}\Big((\kappa_1+\kappa)x_1+2(\kappa_1-\kappa)x_1k(x)\Big).
\end{equation}
We shall use the same argument as in the proof of Proposition \ref{propu11} to establish the estimates of $|\nabla{\bf u}_2|$ and $p_{2}$.  Firstly, a direct calculation yields that for $x\in \Omega_{2R}$,
\begin{equation}\label{estv212}
\partial_{x_1}{\bf v}_2^{(1)}=6\Big(\frac{1}{\delta(x_1)}-\frac{2(\kappa_1-\kappa)x_1^2}{\delta^2(x_1)}\Big)\Big(k^2(x)-\frac{1}{4}\Big)+\frac{12x_1}{\delta(x_1)} k(x)\partial_{x_1}k(x), 
\end{equation}
\begin{equation}\label{estv2112}
\partial_{x_2}{\bf v}_2^{(1)}=\frac{12x_1}{\delta^2(x_{1})}k(x),
\quad
|\partial_{x_1}{\bf v}_2^{(2)}|\leq \frac{C|x_1|}{\delta(x_1)},
\end{equation}
and
\begin{align}\label{estv11a2223}
\partial_{x_2}{\bf v}_2^{(2)}=\frac{1}{\delta(x_1)}+\frac{2}{\delta(x_1)}\bigg(\frac{6(\kappa_1-\kappa) x_{1}^2}{\delta(x_1)}-1\bigg)\Big(k^2(x)-\frac{1}{4}\Big)
+\frac{2k(x)}{\delta(x_1)}G_2(x).
\end{align}
Then, it is easy to check that 
$$\nabla\cdot {\bf v}_{2}=0\quad\mbox{in}~ \Omega_{2R},$$
and 
\begin{align*}%\label{u2v2}
|\nabla{\bf v}_2|\leq C\left(\frac{1}{\delta(x_1)} +\frac{|x_1|}{\delta^2(x_1)}\right)\quad\mbox{in}~ \Omega_{2R}.
\end{align*}

A further calculation gives  
$$|\partial_{x_2x_2}{\bf v}_2^{(1)}|, |\partial_{x_2x_2}{\bf v}_2^{(2)}|\leq \frac{C|x_1|}{\delta^3(x_1)}\quad\mbox{in}~\Omega_{2R},$$
which are not ``good" terms. Hence, we construct  $\bar{p}_2\in C^{1}(\Omega)$, in $\Omega_{2R}$,
\begin{equation}\label{barp2}
\bar{p}_2=-\frac{1}{\kappa_1-\kappa}\frac{3\mu}{\delta^2(x_1)}+\frac{2\mu}{\delta(x_1)}\left(\Big(\frac{6(\kappa_1-\kappa)x_1^2}{\delta(x_1)}-1\Big)k^{2}(x)+k(x)G_2(x)\right),
\end{equation} 
which makes
\begin{align}\label{estf13}
|{\bf f}_{2}|=\left|\mu\,\Delta{\bf v}_{2}-\nabla\bar{p}_2\right|\leq\frac{C|x_1|}{\delta^{2}(x_1)}\quad \mbox{in}~\Omega_{2R}.
\end{align}

Indeed, first note that 
\begin{equation*}%\label{eqv13-p}
\mu\,\partial_{x_2}{\bf v}_2^{(2)}-\bar{p}_2=\frac{3\mu}{(\kappa_1-\kappa)}\frac{1}{\delta^2(x_1)},
\end{equation*}
then after differentiating it with respect to $x_{2}$, we have
\begin{align*}%\label{estv133p13}
\mu\,\partial_{x_2x_2} {\bf v}_2^{(2)}-\partial_{x_2}\bar{p}_2=0.
\end{align*}
On the other hand,
\begin{align}
|\partial_{x_1x_1}{\bf v}_2^{(1)}|&\leq\frac{C|x_1|}{\delta^2(x_{1})},\quad\partial_{x_2x_2}{\bf v}_2^{(1)}=\frac{12x_1}{\delta^3(x_{1})},\label{v11-22-2D}\\
|\partial_{x_1x_1}{\bf v}_2^{(2)}|&\leq\frac{C}{\delta(x_{1})},\quad|\partial_{x_2x_2}{\bf v}_2^{(2)}|\leq\frac{C|x_1|}{\delta^3(x_{1})}\nonumber.
\end{align}
Denoting
$$ \tilde{p}_2 :=\frac{2\mu}{\delta(x_1)}\left(\frac{6(\kappa_1-\kappa)x_1^2}{\delta(x_1)}-1\right)k^{2}(x)+\frac{2\mu k(x)}{\delta(x_1)}G_2(x)\quad\mbox{in}~\Omega_{2R},$$
then we have
$$ |\partial_{x_1} \tilde{p}_2|\leq \frac{C|x_1|}{\delta^2(x_1)}.$$
Thus,
\begin{equation*}
|\mu\,\partial_{x_2x_2} {\bf v}_2^{(1)}- \partial_{x_1}\bar{p}_2|=|\partial_{x_1} \tilde{p}_2|\leq \frac{C|x_1|}{\delta^2(x_1)}.
\end{equation*}
So \eqref{estf13} holds. 

As a consequence,
\begin{align}\label{L2_f13}
\int_{\Omega_{s}(z_1)}|{\bf f}_{2}|^{2}\mathrm{d}x\leq\frac{Cs}{\delta^{3}(z_1)}(s^2+|z_1|^2).
\end{align}
Then, we have the boundedness of the global energy of ${\bf w}_2$.

\begin{lemma}\label{lem_energyw13}
	Let $({\bf w}_2,q_2)$ be the solution to \eqref{w1}. Then
	\begin{align*}%\label{energyw13}
	\int_{\Omega}|\nabla {\bf w}_{2}|^{2}\mathrm{d}x\leq C.
	\end{align*}
\end{lemma}
\begin{proof}
	As in the proof of Lemma \ref{lem3.0}, it suffices to prove
	\begin{align}\label{energyDw13}
	\Big| \int_{\Omega_{R}}\sum_{j=1}^{2}{\bf f}_2^{(j)}{\bf w}_2^{(j)}\mathrm{d}x\Big|\leq\,C\left(\int_{\Omega}|\nabla {\bf w}_{2}|^2\mathrm{d}x\right)^{1/2}.
	\end{align}
	
	To this end, we first note from \eqref{estf13} that
	\begin{equation}\label{energy1}
	\int_{\Omega_{R}}{\bf f}_2^{(1)}{\bf w}_2^{(1)}\mathrm{d}x=\int_{\Omega_{R}}{\bf w}_2^{(1)}(\mu \Delta {\bf v}_2^{(1)}-\partial _{x_1} \bar{p}_2)\mathrm{d}x.
	\end{equation}
	One observes from \eqref{estv212} and \eqref{v11-22-2D} that ${\bf f}_2^{(1)}$ can be rewritten in the polynomial form
	$$\mu \Delta {\bf v}_2^{(1)}-\partial _{x_1} \bar{p}_2:=A^{2}_{0}(x_1)+A^{2}_{1}(x_1) x_2+A^{2}_{2}(x_1)x_{2}^{2},$$
	where $A^{2}_{l}(x_1)$ are polynomials of $x_{1}$, and in view of \eqref{estf13},  
	$$\Big|A^{2}_{0}(x_1)+A^{2}_{1}(x_1) x_2+A^{2}_{2}(x_1)x_{2}^{2}\Big|\leq\frac{C|x_1|}{\delta^{2}(x_1)}.$$
Thus, \eqref{energy1} becomes
	\begin{equation*}
	\int_{\Omega_{R}}{\bf f}_2^{(1)}{\bf w}_2^{(1)}\mathrm{d}x=\int_{\Omega_{r_{0}}}{\bf w}_2^{(1)}(A^{2}_{0}(x_1)+A^{2}_{1}(x_1) x_2+A^{2}_{2}(x_1)x_{2}^{2})\mathrm{d}x.
	\end{equation*}	   	
Because
	$$A^{2}_{0}(x_1)+A^{2}_{1}(x_1) x_2+A^{2}_{2}(x_1)x_{2}^{2}=\partial_{x_2}\bigg(A^{2}_{0}(x_1)x_2+\frac{1}{2} A^{2}_{1}(x_1) x_2^2+\frac{1}{3}A^{2}_{2}(x_1)x_{2}^{3} \bigg),$$
it is easy to get
	$$\int_{\Omega_{R}}\big|A^{2}_{0}(x_1)x_2+\frac{1}{2} A^{2}_{1}(x_1) x_2^2+\frac{1}{3}A^{2}_{2}(x_1)x_{2}^{3}\big|^2\mathrm{d}x\leq C\int_{\Omega_{R}}\frac{|x_1|^2}{\delta^2(x_1)}\mathrm{d}x\leq C.$$
	Since ${\bf w}_2=0$ on $\partial D\cup \partial D_1$, then by applying the integration by parts with respect to $x_{2}$ and using H\"older's inequality and \eqref{w1Dw1}, we have
	\begin{align}\label{ener1}
	\left|\int_{\Omega_{R}} {\bf f}_2^{(1)}{\bf w}_{2}^{(1)}\mathrm{d}x\right|&\leq\int_{\Omega_{R}}|\partial_{x_{2}}{\bf w}_2^{(1)}|\Big|A^{2}_{0}(x_1)x_2+\frac{1}{2} A^{2}_{1}(x_1) x_2^2+\frac{1}{3}A^{2}_{2}(x_1)x_{2}^{3}\Big|\mathrm{d}x\nonumber\\
	&\quad+C\left(\int_{\Omega}|\nabla {\bf w}_{2}|^2\mathrm{d}x\right)^{1/2}\nonumber\\
	&\leq\,C\left(\int_{\Omega}|\nabla {\bf w}_{2}|^2\mathrm{d}x\right)^{1/2}.
	\end{align}

For $j=2$, in view of \eqref{estv2112},
	$$\int_{\Omega_{R}}|\partial_{x_1}{\bf v}_2^{(2)}|^{2}\mathrm{d}x\leq\,C\int_{\Omega_{R}}\frac{|x'|^{2}}{\delta^{2}(x')}\leq\,C.$$  
	By using the integration by parts with respect to $x_1$, we have 
	\begin{align}\label{estf1333}
	&\left|\int_{\Omega_{R}} {\bf f}_2^{(2)}{\bf w}_2^{(2)}\mathrm{d}x\right|
	=\left|\int_{\Omega_{R}} {\bf w}_2^{(2)}(\mu\,\partial_{x_1x_1}{\bf v}_2^{(2)})\mathrm{d}x\right|\nonumber\\
	\leq&\,\int_{\Omega_{R}}|\partial_{x_1}{\bf w}_2^{(2)}||\mu\,\partial_{x_1}{\bf v}_2^{(2)}|\mathrm{d}x+\int_{\substack{|x'|=R,\\h(x')<x_{2}<\epsilon+h_{1}(x_1)}}|{\bf w}_2^{(2)}|\mathrm{d}x_1\nonumber\\
	\leq&\, C\left(\int_{\Omega}|\nabla {\bf w}_2|^2\mathrm{d}x\right)^{1/2}.
	\end{align}
	Combining \eqref{ener1} and \eqref{estf1333}, we derive \eqref{energyDw13}. This completes the proof.
\end{proof}

\begin{lemma}\label{lem3.2}
Let $({\bf w}_2,q_2)$ be the solution to \eqref{w1}. Then
\begin{align*}%\label{estw11narrow2dd}
\int_{\Omega_{\delta}(z_1)}|\nabla {\bf w}_2|^{2}\mathrm{d}x\leq C\delta(z_1).
\end{align*}
\end{lemma}
The proof is very similar to that of Lemma \ref{lem3.1} and thus is omitted here. Then, by a similar argument used in the proof of Proposition \ref{propu11}, together with Lemma \ref{lem3.2}, we obtain 
\begin{prop}\label{propu12}
Let ${\bf u}_{2}\in{C}^{2}(\Omega;\mathbb R^2),~p_{2}\in{C}^{1}(\Omega)$ be the solution to \eqref{equ_v1}. Then there holds
\begin{equation*}
\|\nabla({\bf u}_{2}-{\bf v}_{2})\|_{L^{\infty}(\Omega_{\delta/2}(x_1))}\leq \frac{C}{\sqrt{\delta(x_1)}},\quad\,x\in\Omega_{R}.
\end{equation*}
and 
\begin{equation*}
	\|\nabla^2 ({\bf u}_{2}-{\bf v}_{2})\|_{L^{\infty}(\Omega_{\delta/2}(x_1))}+\|\nabla q_2\|_{L^{\infty}(\Omega_{\delta/2}(x_1))}\leq C\left(\frac{1}{\delta(x_1)}+\frac{|x_1|}{\delta^2(x_1)}\right).
\end{equation*}
Consequently, \begin{align*}
|\nabla {\bf u}_2(x)|\leq C\left(\frac{1}{\delta(x_1)} +\frac{|x_1|}{\delta^2(x_1)}\right),\quad\,x\in\Omega_{R},
\end{align*}
	and
\begin{align*}
	\|p_{2}-(q_{2})_{\Omega_R}\|_{L^{\infty}(\Omega_{\delta/2}(x_{1}))}\leq\,\frac{C}{\varepsilon^{2}},\quad\,x\in\Omega_{R}.
\end{align*}
\end{prop}

\subsection{Estimates of $|\nabla {\bf u}_3|$ and $p_3$}\label{subsec2.3}

Following the above idea, we seek ${\bf v}_{3}\in C^{2}(\Omega;\mathbb R^2)$, satisfying, in $\Omega_{2R}$,
\begin{align}\label{v13}
{\bf v}_{3}=\boldsymbol\psi_{3}\Big(k(x)+\frac{1}{2}\Big)
+\begin{pmatrix}
\frac{-4x_{1}^2}{\delta(x_{1})}-2k(x)x_2-\frac{3x_2^2}{\delta(x_1)}+\frac{1}{\kappa_1-\kappa}\\\\
G_3(x)\end{pmatrix}
\Big(k^2(x)-\frac{1}{4}\Big),
\end{align}
where
\begin{align}\label{G3}
G_3(x)=2x_1k(x)+x_1\Big((\kappa_1+\kappa)+2(\kappa_1-\kappa)k(x)\Big)\left(\frac{-4x_{1}^2}{\delta(x_{1})}-\frac{3x_2^2}{\delta(x_1)}+\frac{1}{\kappa_1-\kappa}\right).
\end{align} 
Choose $\bar{p}_3\in C^{1}(\Omega)$ such that, in $\Omega_{2R}$,
\begin{equation}\label{barp3}
\bar{p}_3=\frac{2\mu }{\kappa_1-\kappa}\frac{ x_1 }{\delta^2(x_1)}+\mu \left(\frac{-x_1}{\delta(x_1)}+ \partial_{x_2}G_3(x)\Big(k^2(x)-\frac{1}{4}\Big)
+\frac{2k(x)}{\delta(x_1)}G_3(x)\right).
\end{equation}

From \eqref{v13} and \eqref{G3}, we have 
\begin{equation*}%\label{g1}
|\partial_{x_1}G_3(x)|\leq C,\quad |\partial_{x_2}G_3(x)|\leq \frac{Cx_1}{\delta(x_1)},
\end{equation*}
and
\begin{align}\label{estv3111}
\partial_{x_1}{\bf v}_3^{(1)}&=x_2\partial_{x_1}k(x)+2k(x)\partial_{x_1}k(x)\left( \frac{-4x_{1}^2}{\delta(x_{1})}-2k(x)x_2-\frac{3x_2^2}{\delta(x_1)}+\frac{1}{\kappa_1-\kappa}\right)\nonumber\\
&\quad-\left(\frac{8x_1}{\delta(x_1)}-\frac{8(\kappa_1-\kappa)x_1^3}{\delta^2(x_1)}+2\partial_{x_1}k(x)x_2-\frac{6(\kappa_1-\kappa)x_1x_2^2}{\delta^2(x_1)}\right)\Big(k^2(x)-\frac{1}{4}\Big) , \\
\partial_{x_2}{\bf v}_3^{(1)}&=k(x)+\frac{1}{2}+\frac{x_2}{\delta(x_{1})}-\Big(2k(x)+\frac{8x_2}{\delta(x_1)}\Big)\Big(k^2(x)-\frac{1}{4}\Big)\nonumber\\&\quad+\frac{2k(x)}{\delta(x_{1})}\left( \frac{-4x_{1}^2}{\delta(x_{1})}-2k(x)x_2-\frac{3x_2^2}{\delta(x_1)}+\frac{1}{\kappa_1-\kappa}\right),\label{estv3112}
\end{align}
\begin{equation}\label{estv3212}
\partial_{x_1}{\bf v}_3^{(2)}=-k(x)-\frac{1}{2}-x_1\partial_{x_1}k(x)+\partial_{x_1}G_3(x)\Big(k^2(x)-\frac{1}{4}\Big)+2k(x)\partial_{x_1}k(x)G_3(x),
\end{equation}
\begin{align}\label{estv111a223}
\partial_{x_2}{\bf v}_3^{(2)}=\frac{-x_1}{\delta(x_1)}+ \partial_{x_2}G_3(x)\Big(k^2(x)-\frac{1}{4}\Big)
+\frac{2k}{\delta(x_1)}G_3(x).
\end{align}
Then
$$\nabla\cdot {\bf v}_{3}=0,\quad\mbox{in}~ \Omega_{2R},$$
and
\begin{align*}%\label{u3v3}
|\nabla{\bf v}_{3}|&\leq \frac{C}{\delta(x_1)},\quad\mbox{in}~ \Omega_{2R}.
\end{align*}

Furthermore,
\begin{align*}
|\partial_{x_1x_1}{\bf v}_3^{(1)}|&\leq\frac{C}{\delta(x_{1})},\quad\quad|\partial_{x_2x_2}{\bf v}_3^{(1)}|\leq\frac{C}{\delta^2(x_{1})},\\%\label{v33-22-2D}
|\partial_{x_1x_1}{\bf v}_3^{(2)}|&\leq\frac{C|x_1|}{\delta(x_{1})},\quad\quad|\partial_{x_2x_2}{\bf v}_3^{(2)}|\leq\frac{C|x_1|}{\delta^2(x_{1})}.%\label{v33-22-2g}
\end{align*}     
Combining with \eqref{barp3} yields 
\begin{equation*}%\label{f32}
\mu\,\partial_{x_2x_2} {\bf v}_3^{(2)}- \partial_{x_2}\bar{p}_3=0,
\end{equation*}
and
\begin{equation*}%\label{f31}
|\mu\,\partial_{x_2x_2} {\bf v}_3^{(1)}- \partial_{x_1}\bar{p}_3|\leq\frac{C}{\delta(x_1)}.
\end{equation*}        
This gives 
\begin{align*}%\label{estf3}
|{\bf f}_{3}|=\left|\mu\,\Delta{\bf v}_{3}-\nabla\bar{p}_3\right|\leq\frac{C}{\delta(x_1)},\quad \mbox{in}~\Omega_{2R}.
\end{align*} 

Similar to the proof of Lemma \ref{lem_energyw13}, we have the estimate of global energy of ${\bf w}_3$. 
\begin{lemma}\label{lem_energyw3}
Let $({\bf w}_3,q_3)$ be the solution to \eqref{w1}. Then
\begin{align}\label{energyw3}
\int_{\Omega}|\nabla {\bf w}_{3}|^{2}\mathrm{d}x\leq C.
\end{align}
\end{lemma}
\begin{proof}
In a similar fashion as in the proof of Lemma \ref{lem_energyw13}, we rewrite $\partial_{x_2x_2}{\bf v}_3^{(1)}-\partial_{x_1}\bar{p}_3$ as follows:
\begin{align*}
\partial_{x_2x_2}{\bf v}_3^{(1)}-\partial_{x_1}\bar{p}_3:&=A_0^3(x_1)+A_1^3(x_1)x_2+A_2^3(x_1)x_2^2+A_3^3(x_1)x_2^3\\&=\partial_{x_2}\big(A_0^3(x_1)x_2+\frac{1}{2}A_1^3(x_1)x_2^2+\frac{1}{3}A_2^3(x_1)x_2^3+\frac{1}{4}A_3^3(x_1)x_2^4\big),
\end{align*}
where $A_l^3(x_1)$ are polunomials of $x_1$.	Moreover, from \eqref{estv3112} and \eqref{barp3}, we have
$$|A_0^3(x_1)x_2+\frac{1}{2}A_1^3(x_1)x_2^2+\frac{1}{3}A_2^3(x_1)x_2^3+\frac{1}{4}A_3^3(x_1)x_2^4|\leq C.$$
Then
\begin{align*}
\left|\int_{\Omega_{R}} (\mu\,\partial_{x_2x_2} {\bf v}_3^{(1)}-\partial_{x_1}\bar{p}_3){\bf w}_3^{(1)}\mathrm{d}x\right|
\leq& C\int_{\Omega_{R}}|\partial_{x_{2}}{\bf w}_3^{(1)}|\mathrm{d}x
\leq C\left(\int_{\Omega}|\nabla {\bf w}_{3}|^2\mathrm{d}x\right)^{1/2}.
\end{align*}
Also, it follows from \eqref{estv3111} that 
$$|\partial_{x_1}{\bf v}_3^{(1)}|\leq\frac{C|x_1|}{\delta(x_1)}.$$
Then
$$ \left|\int_{\Omega_{R}} (\mu\,\Delta{\bf v}_3^{(1)}-\partial_{x_1}\bar{p}_3){\bf w}_3^{(1)}\mathrm{d}x\right|	\leq C\left(\int_{\Omega}|\nabla {\bf w}_{3}|^2\mathrm{d}x\right)^{1/2}.$$

Since
$$ \mu\,\partial_{x_2x_2} {\bf v}_3^{(2)}- \partial_{x_2}\bar{p}_3=0,\quad|\partial_{x_1}{\bf v}_3^{(2)}|\leq C,$$ 
then by using the integration by parts with respect to $x_2$, we obtain
\begin{align*}
\left|\int_{\Omega_{R}} (\mu\,\Delta {\bf v}_3^{(2)}-\partial_{x_2}\bar{p}_3){\bf w}_3^{(1)}\mathrm{d}x\right|
\leq& C\int_{\Omega_{R}}|\partial_{x_{1}}{\bf w}_3^{(2)}|\mathrm{d}x
\leq C\left(\int_{\Omega}|\nabla {\bf w}_{3}|^2\mathrm{d}x\right)^{1/2}.
\end{align*}
We thus have \eqref{energyw3}. This completes the proof.
\end{proof}

As before,  we have 
\begin{prop}\label{propu13}
Let ${\bf u}_{3}\in{C}^{2}(\Omega;\mathbb R^2),~p_{3}\in{C}^{1}(\Omega)$ be the solution to \eqref{equ_v1}. Then 
\begin{equation*}
\|\nabla({\bf u}_{3}-{\bf v}_{3})\|_{L^{\infty}(\Omega_{\delta/2}(x_1))}\leq C,\quad x\in\Omega_{R},
\end{equation*}
and 
\begin{equation*}
	\|\nabla^2 ({\bf u}_{3}-{\bf v}_{3})\|_{L^{\infty}(\Omega_{\delta/2}(x_1))}+\|\nabla q_3\|_{L^{\infty}(\Omega_{\delta/2}(x_1))}\leq \frac{C}{\delta(x_1)}.
\end{equation*}
Consequently, for $x\in\Omega_{R}$,
\begin{align*}
|\nabla {\bf u}_3(x)|\leq  \frac{C}{\delta(x_1)},\quad\,x\in\Omega_{R},
\end{align*}
and
\begin{align*}
	\|p_{3}-(q_{3})_{\Omega_R}\|_{L^{\infty}(\Omega_{\delta/2}(x_{1}))}\leq\,\frac{C}{\varepsilon^{3/2}},\quad\,x\in\Omega_{R}.
\end{align*}
\end{prop}

Based on these estimates for $\nabla {\bf u}_\alpha$ and $p_{\alpha}$, $\alpha=1,2,3$, established in Proposition \ref{propu11}, Proposition \ref{propu12} and Proposition \ref{propu13}, we will first prove the invertibility of the matrix $(a_{\alpha\beta})_{3\times3}$ in the next section. 

\section{Asyptotics and Estimates of $a_{\alpha\beta}$}\label{estCalpha}

In order to solve the free constants $C^\alpha$, $\alpha=1,2,3$, we first establish the asymptotics and estimates of $a_{\alpha\beta}$ in this section, and show the matrix $(a_{\alpha\beta})_{3\times3}$ is invertible.

\begin{lemma}\label{lema114D}
	We have
	\begin{align}
	a_{11}&=
	\frac{\mu\,\pi}{\sqrt{\kappa_1-\kappa}}\bigg(1+\frac{(\kappa_1+\kappa)^2}{(\kappa_1-\kappa)^{2}}+\frac{\kappa(\kappa_1+\kappa)^2}{3(\kappa_1-\kappa)^{3}}
	\bigg)\frac{1}{\sqrt{\varepsilon}}+O(1),\\
	a_{22}
	&=\frac{3\mu\,\pi}{2(\kappa_1-\kappa)^{3/2}}\frac{1}{\varepsilon^{3/2}}+\frac{O(1)}{\sqrt{\varepsilon}},\label{a22}\\
	a_{33}
	&=\frac{\mu \pi }{3(\kappa_1-\kappa)^{5/2}}\bigg(3+\frac{\kappa}{\kappa_1-\kappa}
	\bigg)\frac{1}{\sqrt{\varepsilon}}+O(1),\\
	a_{13}
	&=\frac{\mu \pi (\kappa_1+\kappa)}{3(\kappa_1-\kappa)^{5/2}}\bigg(3+\frac{\kappa}{\kappa_1-\kappa}
	\bigg)\frac{1}{\sqrt{\varepsilon}}+O(1).
	\end{align}
\end{lemma}
\begin{proof}
For $a_{11}$, it follows from \eqref{defaij}, Proposition \ref{propu11} and \eqref{estv112}--\eqref{estv11223}  that
\begin{align}\label{esa11}
a_{11}&=\int_{\Omega_R} \left(2\mu e({\bf u}_{1}), e({\bf u}_1)\right)\mathrm{d}x+O(1)\nonumber\\
&=\int_{\Omega_R} \left(2\mu e({\bf v}_1), e({\bf v}_1)\right)\mathrm{d}x+\int_{\Omega_R} \left(2\mu e({\bf v}_1), e({\bf w}_1)\right)\mathrm{d}x\nonumber\\
&\quad+\int_{\Omega_R} \left(2\mu e({\bf w}_1), e({\bf v}_1)\right)\mathrm{d}x+\int_{\Omega_R} \left(2\mu e({\bf w}_1), e({\bf w}_1)\right)\mathrm{d}x+O(1)\nonumber\\
&=\int_{\Omega_R} \left(2\mu e({\bf v}_1), e({\bf v}_1)\right)\mathrm{d}x+O(1).
\end{align}
Using \eqref{estv112}--\eqref{estv11223} again, one can verify that the biggest term in $e({\bf v}_1)$ is $\partial_{x_2}{\bf v}_1^{(1)}$, and 
\begin{align*}
\int_{\Omega_R} \left(\partial_{x_2}{\bf v}_1^{(1)}\right)^2\mathrm{d}x
&=\int_{\Omega_R}\left(\frac{1}{\delta(x_{1})}+\frac{2k(x)}{\delta(x_{1})}\Big(-4(\kappa_1+\kappa)\frac{x_1^2}{\delta(x_{1})}+\frac{\kappa_1+\kappa}{\kappa_1-\kappa}\Big)\right)^2\mathrm{d}x\\
&=\frac{\pi}{\sqrt{\kappa_1-\kappa}}\bigg(1+\frac{(\kappa_1+\kappa)^2}{(\kappa_1-\kappa)^{2}}+\frac{\kappa(\kappa_1+\kappa)^2}{3(\kappa_1-\kappa)^{3}}
\bigg)\frac{1}{\sqrt{\varepsilon}}+O(1).
\end{align*}
Thus, 
\begin{align*}
a_{11}
&=\mu\,\int_{\Omega_R} \left(\partial_{x_2}{\bf v}_1^{(1)}\right)^2\mathrm{d}x+O(1)\\
&=\frac{\mu\,\pi}{\sqrt{\kappa_1-\kappa}}\bigg(1+\frac{(\kappa_1+\kappa)^2}{(\kappa_1-\kappa)^{2}}+\frac{\kappa(\kappa_1+\kappa)^2}{3(\kappa_1-\kappa)^{3}}
\bigg)\frac{1}{\sqrt{\varepsilon}}+O(1).
\end{align*}

From Proposition \ref{propu12} and \eqref{estv212}--\eqref{estv11a2223}, we have
\begin{align*}
a_{22}=\mu\,\int_{\Omega_R} \left(\partial_{x_2}{\bf v}_1^{(2)}\right)^2\mathrm{d}x+\frac{O(1)}{\sqrt{\varepsilon}}=\frac{3\mu\,\pi}{2(\kappa_1-\kappa)^{3/2}}\frac{1}{\varepsilon^{3/2}}+\frac{O(1)}{\sqrt{\varepsilon}}.
\end{align*}

Similarly, from Proposition \ref{propu13} and \eqref{estv3111}--\eqref{estv111a223}, we have 
	\begin{align*}
	a_{33}&=\int_{\Omega_R} \left(2\mu e({\bf v}_3), e({\bf v}_3)\right)\mathrm{d}x+O(1)\\
	&=\mu\,\int_{\Omega_R} \left(\partial_{x_2}{\bf v}_3^{(1)}\right)^2\mathrm{d}x+O(1)\\
	&=\mu\,\int_{\Omega_R}\left(\frac{2k(x)}{\delta(x_{1})}\left( \frac{-4x_{1}^2}{\delta(x_{1})}+\frac{1}{\kappa_1-\kappa}\right)\right)^2\mathrm{d}x+O(1)\\
	&=\frac{\mu \pi }{3(\kappa_1-\kappa)^{5/2}}\bigg(3+\frac{\kappa}{\kappa_1-\kappa}
	\bigg)\frac{1}{\sqrt{\varepsilon}}+O(1).
	\end{align*} 

Combining Proposition \ref{propu11} and Proposition \ref{propu13}, \eqref{estv112}--\eqref{estv11223} and \eqref{estv3111}--\eqref{estv111a223}, we deduce 
	\begin{align}\label{esta13}
	&a_{13}
	=\int_{\Omega_R} \left(2\mu e({\bf v}_1), e({\bf v}_3)\right)\mathrm{d}x+O(1)\nonumber\\
	&=\mu\,\int_{\Omega_R}\frac{2k(x)}{\delta(x_{1})}\left( \frac{-4x_{1}^2}{\delta(x_{1})}+\frac{1}{\kappa_1-\kappa}\right)\left(\frac{1}{\delta(x_{1})}+\frac{2k(x)}{\delta(x_{1})}\Big(\frac{-4(\kappa_1+\kappa)x_1^2}{\delta(x_{1})}+\frac{\kappa_1+\kappa}{\kappa_1-\kappa}\Big)\right)\mathrm{d}x\nonumber\\
	&\quad+O(1)\nonumber\\
	&=\frac{\mu \pi}{3}\frac{\kappa_1+\kappa}{(\kappa_1-\kappa)^{5/2}}\bigg(3+\frac{\kappa}{\kappa_1-\kappa}
	\bigg)\frac{1}{\sqrt{\varepsilon}}+O(1).
	\end{align} 
	The Lemma is proved.
\end{proof}  

\begin{lemma}\label{lema113D}
	We have
	\begin{align}
	|a_{12}|=|a_{21}|\leq  C|\ln \epsilon|,\quad\mbox{and}~
	|a_{23}|=|a_{32}|\leq C|\ln \epsilon|.\label{esta1112}
	\end{align}
\end{lemma}

\begin{proof}
	By virtue of \eqref{defaij}, we have
	\begin{align}\label{deca12}
		a_{12}&=\int_{\Omega_R} \left(2\mu e({\bf u}_{1}), e({\bf u}_2)\right)\mathrm{d}x+O(1)\nonumber\\
		&=\int_{\Omega_R} \left(2\mu e({\bf v}_1), e({\bf v}_2)\right)\mathrm{d}x+\int_{\Omega_R} \left(2\mu e({\bf v}_1), e({\bf w}_2)\right)\mathrm{d}x\nonumber\\
		&\quad+\int_{\Omega_R} \left(2\mu e({\bf w}_1), e({\bf v}_2)\right)\mathrm{d}x+\int_{\Omega_R} \left(2\mu e({\bf w}_1), e({\bf w}_2)\right)\mathrm{d}x+O(1).
	\end{align}
	It follows from Propositions \ref{propu11} and \ref{propu12}, \eqref{estv112}--\eqref{estv11223} and \eqref{estv212}--\eqref{estv2112} that
	\begin{align*}
		\left|\int_{\Omega_R} \left(2\mu e({\bf v}_{1}), e({\bf w}_{2})\right)\mathrm{d}x\right|\leq\int_{|x_1|\leq R}\frac{C}{\sqrt{\delta(x_1)}}\ dx_1\leq C|\ln\varepsilon|,
	\end{align*}
	\begin{align*}
		\left|\int_{\Omega_R} \left(2\mu e({\bf w}_{1}), e({\bf v}_{2})\right)\mathrm{d}x\right|\leq\int_{|x_1|\leq R}\frac{C|x_1|}{\delta(x_1)}\ dx_1\leq C|\ln\varepsilon|,
	\end{align*}
	and 
	\begin{align*}
		\left|\int_{\Omega_R} \left(2\mu e({\bf w}_{1}), e({\bf w}_{2})\right)\mathrm{d}x\right|\leq \int_{|x_1|\leq R}C\sqrt{\delta(x_1)}\ dx_1\leq C.
	\end{align*}
  Moreover, one can see that  the biggest term in $(2\mu e({\bf v}_{1}), e({\bf v}_{2}))$ is
	\begin{align*}
		\partial_{x_2}{\bf v}_1^{(1)}\cdot\partial_{x_2}{\bf v}_2^{(1)}=\frac{12x_1}{\delta^2(x_1)}k(x)\left(\frac{1}{\delta(x_{1})}+\frac{2(\kappa_1+\kappa)k(x)}{\delta(x_{1})}\Big(\frac{-4x_1^2}{\delta(x_{1})}+\frac{1}{\kappa_1-\kappa}\Big)\right),
	\end{align*} 
	which is an odd function with respect to $x_1$ and thus the integral is $0$. The integral of the rest terms is bounded by $C|\ln\varepsilon|$. Hence, coming back to \eqref{deca12}, we derive
	$$|a_{12}|=|a_{21}|\leq C|\ln\varepsilon|.$$

Similarly, for $a_{23}$, recalling \eqref{estv212}--\eqref{estv11a2223} and  \eqref{estv3111}--\eqref{estv111a223}, we have 
	 $$|a_{32}|=|a_{23}|\leq  C|\ln\varepsilon|.$$ 
Thus, \eqref{esta1112} is proved. We complete the proof of the lemma.
\end{proof} 

By virtue of Lemma \ref{lema114D} and Lemma \ref{lema113D}, we have
$$(a_{11}a_{33}-a_{13}^2)=\frac{(\mu\,\pi)^2}{3(\kappa_1-\kappa)^3}\bigg(1+\frac{2\kappa_1-\kappa}{\kappa_1-\kappa}
\bigg)\frac{1}{\varepsilon}+\frac{O(1)}{\sqrt{\varepsilon}},$$
then, in view of \eqref{a22},
\begin{align}\label{DetA}
\det\mathbb{A}&=a_{22}(a_{11}a_{33}-a_{13}^2)+\frac{O(1)}{\sqrt{\varepsilon}}=\frac{(\mu\,\pi)^3}{2(\kappa_1-\kappa)^{9/2}}\bigg(1+\frac{2\kappa_1-\kappa}{\kappa_1-\kappa}
\bigg)\frac{1}{\varepsilon^{5/2}}+\frac{O(1)}{\varepsilon^{3/2}}.
\end{align}
As a consequence, we have $\mathbb{A}:=(a_{\alpha\beta})_{3\times3}$ is invertible. Recalling \eqref{ce}, by using Crame's law, we have
\begin{equation}\label{defma}
C^\alpha=\frac{\det\mathbb{A}_{\alpha}}{\det\mathbb{A}},
\end{equation} 
where $\mathbb{A}_{\alpha}$ is a matrix after replacing $\alpha$-th column of $\mathbb{A}$ with $(Q_1[\boldsymbol\varphi],Q_2[\boldsymbol\varphi],Q_3[\boldsymbol\varphi])^{\mathrm T}$, that is,
\begin{equation*}%\label{Mtr}
\mathbb{A}_{1}=\begin{pmatrix}
Q_1[\boldsymbol\varphi]&a_{12}&a_{13}\\
Q_2[\boldsymbol\varphi]&a_{22}&a_{23}\\
Q_3[\boldsymbol\varphi]&a_{32}&a_{33}
\end{pmatrix},~\mathbb{A}_{2}=\begin{pmatrix}
a_{11}&Q_1[\boldsymbol\varphi]&a_{13}\\
a_{21}&Q_2[\boldsymbol\varphi]&a_{23}\\
a_{31}&Q_3[\boldsymbol\varphi]&a_{33}
\end{pmatrix},~\mathbb{A}_{3}=\begin{pmatrix}
a_{11}&a_{12}&Q_1[\boldsymbol\varphi]\\
a_{21}&a_{22}&Q_2[\boldsymbol\varphi]\\
a_{31}&a_{32}&Q_3[\boldsymbol\varphi]
\end{pmatrix}.
\end{equation*}
Hence, to solve for $C^\alpha$, we only need to estimate $Q_\beta[\boldsymbol\varphi]$ in the next section.

\section{Locally Constant Boundary Data Case} \label{estu0}

The section is devoted to studying the effect from the locally constant boundary data on the solution and proving Theorem \ref{mainthm0}. We will take the case ${\boldsymbol{\varphi}}\in{\bf\Phi}_{1}$ for instance and give more details for it, while for the case ${\boldsymbol{\varphi}}\in{\bf\Phi}_{2}$ we just list the main differences in subsection \ref{subsec4.5}, where ${\bf\Phi}_{1}$ and ${\bf\Phi}_{2}$ are defined in \eqref{defphi1} and \eqref{defphi2}, respectively.

\subsection{Estimates of $|\nabla {\bf u}_0|$ and $p_0$ for the case ${\boldsymbol{\varphi}}\in{\bf\Phi}_{1}$}
We choose ${\bf v}_{0}^1\in C^{2}(\Omega;\mathbb R^2)$ satisfying, in $\Omega_{2R}$,
\begin{align*}%\label{v05alpha}
{\bf v}_{0}^1=\boldsymbol\varphi\Big(\frac{1}{2}-k(x)\Big)+(\kappa_1+\kappa)\begin{pmatrix}
\frac{4x_1^2}{\delta(x_{1})}-\frac{1}{\kappa_1-\kappa}\\\\
G_0^1(x)
\end{pmatrix}
\Big(k^2(x)-\frac{1}{4}\Big)\quad \mbox{in}~\Omega_{2R},
\end{align*}
where
\begin{align*}
G_0^1(x)=-G_1(x)
\end{align*}
and $G_1(x)$ is defined in \eqref{defG1}. Moreover, 
\begin{equation}\label{v01v1}
{\bf v}_{0}^1+{\bf v}_1=(1,0)^{\mathrm T}\quad \mbox{in}~\Omega_{2R}.
\end{equation} 
We choose $\bar{p}_0^1\in C^{1}(\Omega)$, such that
$\bar{p}_0^1=-\bar{p}_1$ in $\Omega_{2R}$.
Thus, Proposition \ref{propu11} holds true for $({\bf u}_0,p_0)$, with $({\bf v}_1,\bar{p}_1)$ replaced by $({\bf v}_0^1,\bar{p}_0^1)$.

\subsection{Asymptotics and estimates of $Q_\beta[{\boldsymbol\varphi}]$ for the case ${\boldsymbol{\varphi}}\in{\bf\Phi}_{1}$}
Next, we are ready to study the asymptotics and estimates of $Q_\beta[{\boldsymbol\varphi}]$, for this kind of locally constant boundary data. 
\begin{prop}\label{lemaQb1}
	If  ${\boldsymbol{\varphi}}\in{\bf\Phi}_{1}$, then we have
\begin{align*}
	Q_{1}[{\boldsymbol\varphi}]
	&=\frac{\mu\,\pi}{\sqrt{\kappa_1-\kappa}}\bigg(1+\frac{(\kappa_1+\kappa)^2}{(\kappa_1-\kappa)^{2}}+\frac{\kappa(\kappa_1+\kappa)^2}{3(\kappa_1-\kappa)^{3}}
	\bigg)\frac{1}{\sqrt{\varepsilon}}+O(1),\\
	Q_{3}[{\boldsymbol\varphi}]
	&=\frac{\mu \pi (\kappa_1+\kappa)}{3(\kappa_1-\kappa)^{5/2}}\bigg(3+\frac{\kappa}{\kappa_1-\kappa}
	\bigg)\frac{1}{\sqrt{\varepsilon}}+O(1),
	\end{align*}
	and
	$$|Q_2[\boldsymbol\varphi]|\leq C|\ln \varepsilon|.$$
\end{prop}

\begin{proof}
Recalling the definition of $Q_\beta[\boldsymbol\varphi]$, by virtue of \eqref{v01v1}, one can verify that the estimates of $|Q_\beta[\boldsymbol\varphi]|$ are very similar to that of $|a_{1\beta}|$ established in Section \ref{estCalpha}. 

Indeed, 
from \eqref{defaij}, \eqref{estv112}--\eqref{estv11223}, Proposition \ref{propu11}, it follows that
\begin{align*}
Q_{1}[{\boldsymbol\varphi}]=-\int_{\Omega_R} \left(2\mu e({\bf v}_{0}^{1}), e({\bf v}_1)\right)\mathrm{d}x+O(1).
\end{align*} 
Using \eqref{v01v1}, we have 
\begin{align*}
Q_{1}[{\boldsymbol\varphi}]=\int_{\Omega_R} \left(2\mu e({\bf v}_{1}), e({\bf v}_1)\right)\mathrm{d}x+O(1),
\end{align*} 
which is the same as $a_{11}$ in \eqref{esa11}. From \eqref{estv3111}--\eqref{estv111a223},  Proposition \ref{propu13}, and \eqref{v01v1}, we have
\begin{align*}
Q_{3}[{\boldsymbol\varphi}]
&=-\int_{\Omega_R} \left(2\mu e({\bf v}_{0}^{1}), e({\bf v}_3)\right)\mathrm{d}x+O(1)=\int_{\Omega_R} \left(2\mu e({\bf v}_{1}), e({\bf v}_3)\right)\mathrm{d}x+O(1).
\end{align*}
Thus, the estimate of $Q_{3}[{\boldsymbol\varphi}]$ follows from \eqref{esta13}. 
Similarly, $|Q_2[{\boldsymbol\varphi}]|\leq C|\ln \varepsilon|$. Proposition \ref{lemaQb1} is proved.
\end{proof}

\subsection{Proof of Proposition \ref{propu18} for ${\boldsymbol{\varphi}}\in{\bf\Phi}_{1}$}

Now we solve $C^\alpha$ and give the estimates for them.

\begin{proof}[Proof of Proposition \ref{propu18}] 

Denote by $\mbox{cof}(\mathbb A)_{\alpha\beta}$  the cofactor of $a_{\alpha\beta}$. Making use of Lemma \ref{lema114D} and Lemma \ref{lema113D}, we have
\begin{align}\label{estcofA11}
\mbox{cof}(\mathbb A)_{11}=a_{22}a_{33}-a_{23}^2=\frac{(\mu \pi)^2 }{2(\kappa_1-\kappa)^{4}}\bigg(3+\frac{\kappa}{\kappa_1-\kappa}
\bigg)\frac{1}{\varepsilon^2}+O(1)\varepsilon^{-3/2},
\end{align}
\begin{align}\label{estcofA31}
\mbox{cof}(\mathbb A)_{31}=a_{12}a_{23}-a_{22}a_{13}=-\frac{(\mu \pi)^2 (\kappa_1+\kappa)}{2(\kappa_1-\kappa)^{4}}\bigg(3+\frac{\kappa}{\kappa_1-\kappa}
\bigg)\frac{1}{\varepsilon^2}+O(1)\varepsilon^{-3/2},
\end{align}
and
$$|\mbox{cof}(\mathbb A)_{21}|\leq \frac{C}{\sqrt{\varepsilon}}.$$
Then by using \eqref{ce}, Cramer's rule,  and Proposition \ref{lemaQb1},  we deduce
\begin{align}\label{estC1}
C^1=\frac{1}{\det\mathbb{A}}\Big(\mbox{cof}(\mathbb A)_{11} Q_{1}[{\boldsymbol\varphi}]-\mbox{cof}(\mathbb A)_{21} Q_{2}[{\boldsymbol\varphi}]+\mbox{cof}(\mathbb A)_{31}Q_{3}[{\boldsymbol\varphi}] \Big).
\end{align}
It follows from Proposition \ref{lemaQb1} that
\begin{align}\label{estQ1Q3}
Q_{1}[{\boldsymbol\varphi}]-(\kappa_1+\kappa)	Q_{3}[{\boldsymbol\varphi}]=\frac{\mu\,\pi}{\sqrt{\kappa_1-\kappa}}\frac{1}{\sqrt{\varepsilon}}
+O(1).
\end{align}	
Substituting \eqref{estcofA11}, \eqref{estcofA31}, \eqref{estQ1Q3}, and \eqref{DetA} into \eqref{estC1}, we obtain
\begin{align}\label{lowerC1}
C^1=1+O(1)\sqrt\varepsilon.
\end{align}

Similarly, 
$$|C^3|\leq C\sqrt{\varepsilon},\quad |C^2|\leq C\varepsilon^{3/2}.$$
This completes the proof of \eqref{lamda0}. 
\end{proof}

\subsection{Proof of Theorem \ref{mainthm0} for ${\boldsymbol{\varphi}}\in{\bf\Phi}_{1}$}\label{mainresult}
With the above estimates at hand, we are ready to complete the proof of Theorem \ref{mainthm0} for ${\boldsymbol{\varphi}}\in{\bf\Phi}_{1}$.

\begin{proof}[Proof of Theorem \ref{mainthm0}.]   
We first prove the case of ${\boldsymbol{\varphi}}\in{\bf\Phi}_{1}$. It follows  from Proposition \ref{propu12}, Proposition \ref{propu13}, and \eqref{lamda0} that
\begin{equation*}
\left|\sum_{\alpha=2}^{3}C^\alpha\nabla{\bf u}_\alpha\right|\leq C\varepsilon^{3/2}\left(\frac{1}{\delta(x_1)}+\frac{|x_1|}{\delta^2(x_1)}\right)+\frac{C\sqrt{\varepsilon}}{\delta(x_1)}\leq\frac{C}{\sqrt{\delta(x_1)}}.
\end{equation*}
From \eqref{estv1112} and \eqref{v01v1}, one notes that 
\begin{equation}\label{v11v01}
|\nabla({\bf v}_1+{\bf v}_{0}^1)|=0.
\end{equation}
Then, combining Proposition \ref{propu11}, Proposition \ref{propu13}, and \eqref{v01v1}, we obtain 
\begin{align}\label{estu0u1}
|\nabla({\bf u}_1+{\bf u}_0)|\leq |\nabla({\bf v}_1+{\bf v}_{0}^1)|+|\nabla({\bf w}_1+{\bf w}_{0}^1)|\leq C.
\end{align}
Hence, using \eqref{lowerC1} and Proposition \ref{propu11}, we deduce
\begin{align*}
|\nabla{\bf u}(x)|=\left|\sum_{\alpha=1}^{3}C^\alpha\nabla{\bf u}_\alpha+\nabla{\bf u}_0\right|
\leq \left|(C^1-1)\nabla{\bf u}_1\right|+\frac{C}{\sqrt{\delta(x_1)}}\leq\frac{C}{\sqrt{\delta(x_1)}}.
\end{align*}
For the estimate of $|p|$, we first recall that
$$p(x)=\sum_{\alpha=1}^{3}C^{\alpha}p_{\alpha}(x)+p_{0}(x).$$
Denote 
$$q:=\sum_{\alpha=1}^{3}C^{\alpha}q_{\alpha}(x)+q_{0}(x),$$
where $q_{\alpha}=p_{\alpha}-\bar p_{\alpha}$. Recalling  $\bar{p}_0^1=-\bar{p}_1$ in $\Omega_{2R}$, one can see that 
$$|q_{0}+q_1-(q_{0}+q_1)_{\Omega_R}|\leq \frac{C}{\varepsilon}.$$
Then together with  the estimate of $|p_\alpha|$ in Proposition \ref{propu11}, Proposition \ref{propu12}, and Proposition \ref{propu13}, and \eqref{lamda0}, we derive 
\begin{align*}
|p(x)-(q)_{\Omega_{R}}|&\leq\Big|C^{1}(p_{1}-(q_{1})_{\Omega_{R}})\Big|+ \sum_{\alpha=2}^{3}\Big|C^{\alpha}(p_{\alpha}-(q_{\alpha})_{\Omega_{R}})\Big|\\
&\quad+|q_{0}+q_1-(q_{0}+q_1)_{\Omega_R}|\leq \frac{C}{\varepsilon}.
\end{align*}
This finishes the proof of Theorem \ref{mainthm0} for ${\boldsymbol{\varphi}}\in{\bf\Phi}_{1}$.
\end{proof}

\subsection{The case ${\boldsymbol{\varphi}}\in{\bf\Phi}_{2}$}\label{subsec4.5}

We choose ${\bf v}_{0}^2\in C^{2}(\Omega;\mathbb R^2)$ satisfying, in $\Omega_{2R}$,
\begin{align*}%\label{v06alpha}
{\bf v}_{0}^2=\boldsymbol\varphi\Big(\frac{1}{2}-k(x)\Big)+\begin{pmatrix}
\frac{-6x_1}{\delta(x_{1})}\\\\
G_0^2(x)
\end{pmatrix}
\Big(k^2(x)-\frac{1}{4}\Big)\quad \mbox{in}~\Omega_{2R},
\end{align*}
where
$$G_0^2(x)=-G_2(x)$$
and $G_2(x)$ is defined in \eqref{defG2}. Thus we have 
\begin{equation*}
{\bf v}_{0}^2+{\bf v}_2=(0,1)^{\mathrm T}.
\end{equation*}
We construct $\bar{p}_0^2\in C^{1}(\Omega)$ such that in $\Omega_{2R}$,
\begin{equation}\label{barp062}
\bar{p}_0^2=\frac{3\mu}{(\kappa_1-\kappa)\delta^2(x_1)}+\mu\,\partial_{x_2}({\bf v}_0^2)^{(2)}.
\end{equation}

By a direct calculation, we derive 
\begin{align*}
|\nabla{\bf v}_0^2|\leq 
\frac{C}{\delta(x_1)}+\frac{C|x_1|}{\delta^2(x_1)}.
\end{align*}
Combining with the definition of $\bar p_0^2$ in \eqref{barp062}, we deduce, in $\Omega_{2R}$,
\begin{align*}
|{\bf f}_0^2|=\left|\mu\,\Delta{\bf v}_0^2-\nabla\bar{p}_0^2\right|\leq
\frac{C|x_1|}{\delta^2(x_1)}.
\end{align*} 
Thus, Proposition \ref{propu12} holds true for $({\bf u}_0,p_0)$ with $({\bf v}_2,\bar{p}_2)$ replaced by $({\bf v}_0^2,\bar{p}_0^2)$.

By mimicking the proof of Proposition \ref{lemaQb1}, we obtain \begin{prop}\label{lemaQb2}
If ${\boldsymbol{\varphi}}\in{\bf\Phi}_{2}$, then	we have
\begin{align*}%\label{QB1}
|Q_\beta[\boldsymbol\varphi]|\leq
C|\ln \varepsilon|,~\beta=1,3,\quad
Q_2[\boldsymbol\varphi]=\frac{3\mu\,\pi}{2}\frac{1}{(\kappa_1-\kappa)^{3/2}\varepsilon^{3/2}}+O(1)\frac{1}{\sqrt{\varepsilon}}.
\end{align*}
\end{prop}

The proof of Proposition \ref{propu18} and Theorem \ref{mainthm0} for the case of ${\boldsymbol{\varphi}}\in{\bf\Phi}_{2}$ is the same as that for ${\boldsymbol{\varphi}}\in{\bf\Phi}_{1}$.

\section{Locally Polynomial Boundary Data Case}\label{sec5}

In this section, we investigate the effect from the locally polynomial boundary data on the solution. We first prove the general Theorem \ref{mainthm}, then prove Theorem \ref{mainthm1} for ${\boldsymbol{\varphi}}\in{\bf\Phi}_{3}$ and Theorem \ref{mainthm2} for ${\boldsymbol{\varphi}}\in{\bf\Phi}_{4}$.

\subsection{Proof of Theorem \ref{mainthm}}

\begin{proof}[Proof of Theorem \ref{mainthm}]  
	Substituting \eqref{defma}, the estimates of $|\nabla{\bf u}_\alpha|$ in Proposition \ref{propu11}, Proposition \ref{propu12}, and Proposition \ref{propu13} into \eqref{introC}, we obtain, for $x\in\Omega_R$,
	\begin{align*}
	|\nabla{\bf u}(x)|\leq C\left(\frac{|\det\mathbb{A}_{1}|}{\det\mathbb{A}}+\frac{|\det\mathbb{A}_{3}|}{\det\mathbb{A}}\right)\frac{1}{\delta(x_1)}+\frac{C|\det\mathbb{A}_{2}|}{\det\mathbb{A}}\left(\frac{1}{\delta(x_1)} +\frac{|x_1|}{\delta^2(x_1)}\right)+|\nabla{\bf u}_0(x)|.
	\end{align*}
	From Lemma \ref{lema114D} and Lemma \ref{lema113D}, we have
	\begin{equation}\label{detAij}
	|\det\mathbb{A}_{1}|,|\det\mathbb{A}_{3}|\leq C\varepsilon^{-2}\Big(|Q_{1}[{\boldsymbol\varphi}]|+|Q_{3}[{\boldsymbol\varphi}]|\Big),\quad |\det\mathbb{A}_{2}|\leq C\varepsilon^{-1}|Q_{2}[{\boldsymbol\varphi}]|.
	\end{equation}
Combining with \eqref{DetA} leads to
	\begin{align*}
	|\nabla{\bf u}(x)|\leq \frac{C\sqrt\varepsilon}{\delta(x_1)}\Big(|Q_{1}[\boldsymbol\varphi]|+|Q_3[\boldsymbol\varphi]|+|x_1||Q_{2}[\boldsymbol\varphi]|\Big)+|\nabla{\bf u}_0(x)|.
	\end{align*}

For $|p|$,  by applying \eqref{defma}, the estimates of $|p|$ in Proposition \ref{propu11}, Proposition \ref{propu12}, and Proposition \ref{propu13}, \eqref{DetA}, and \eqref{detAij}, we derive 
	\begin{align*}
	|p(x)-(q)_{\Omega_{R}}|&\leq \sum_{\alpha=1}^{3}|C^{\alpha}(p_{\alpha}-(q_{\alpha})_{\Omega_{R}})+p_{0}-(q_{0})_{\Omega_{R}}|\\
	&\leq \frac{C}{\varepsilon^{3/2}}\left(\frac{|\det\mathbb{A}_{1}|}{\det\mathbb{A}}+\frac{|\det\mathbb{A}_{3}|}{\det\mathbb{A}}\right)+\frac{C}{\varepsilon^{2}}\frac{|\det\mathbb{A}_{2}|}{\det\mathbb{A}}+|p_{0}-(q_{0})_{\Omega_R}|\\
	&\leq \frac{C}{\varepsilon}\big(|Q_{1}[\boldsymbol\varphi]|+|Q_3[\boldsymbol\varphi]|\big)+\frac{C|Q_{2}[\boldsymbol\varphi]|}{\sqrt{\varepsilon}}+|p_{0}-(q_{0})_{\Omega_R}|.
	\end{align*}
	The proof of Theorem \ref{mainthm} is completed.
\end{proof}

\subsection{Proof of Theorem \ref{mainthm1}}\label{seccoro}

If ${\boldsymbol{\varphi}}\in{\bf\Phi}_{3}$, then we seek ${\bf v}_{0}^3\in C^{2}(\Omega;\mathbb R^2)$ satisfying, in $\Omega_{2R}$,
\begin{align*}%\label{v01alpha}
{\bf v}_0^3=\boldsymbol\varphi\Big(\frac{1}{2}-k(x)\Big)+\begin{pmatrix}
\frac{32(\kappa_1-\kappa)k(x)+12\kappa}{(l_1+2)}\frac{ x_1^{l_1+2}}{\delta(x_{1})}-\Big(8k(x)-3\Big)x_1^{l_1}\\\\
G_0^3(x)
\end{pmatrix}
\Big(k^2(x)-\frac{1}{4}\Big),
\end{align*}
where
\begin{align*}
G_0^3(x)&=(2k(x)-1)k(x)\delta(x_1)l_1x_1^{l_1-1}-2x_1^{l_1+1}k(x)\Big((\kappa_1+\kappa)+2(\kappa_1-\kappa)k(x)\Big)\\&\quad-\partial_{x_1}k(x)\left(\frac{32(\kappa_1-\kappa)k(x)+12\kappa}{(l_1+2)} x_1^{l_1+2}-\Big(8k(x)-3\Big)x_1^{l_1}\delta(x_1)\right).
\end{align*} 
Notice that the biggest term is $\partial_{x_2x_2}({\bf v}_0^3)^{(1)}\sim \frac{1}{\delta(x_1)}$. This implies that we can choose  $\bar{p}_0^3=0$  in $\Omega$. Now by a direct calculation, we have 
\begin{align*}%\label{v01x1est}
|\partial_{x_1}({\bf v}_0^3)^{(1)}|\leq\frac{|x_1|^{l_1+1}}{\delta(x_1)},\quad|\partial_{x_2}({\bf v}_0^3)^{(1)}|\leq\frac{|x_1|^{l_1}}{\delta(x_1)},
\end{align*}
\begin{align*}%\label{v02x1est1}
|\partial_{x_2}G_0^3(x)|\leq C|x_1^{l_1-1}|,\quad |\partial_{x_1}({\bf v}_0^3)^{(2)}|\leq C|x_1^{l_1}|,
\end{align*}
\begin{align*}%\label{v02x1est}
\partial_{x_2}({\bf v}_0^3)^{(2)}=\partial_{x_2}G_0^3(x)\Big(k^2(x)-\frac{1}{4}\Big)+G_0^3(x)\frac{2k(x)}{\delta(x_1)}.
\end{align*}
Then we have
\begin{align*}%\label{u01v01}
|\nabla{\bf v}_0^3|\leq \frac{C|x_1|^{l_1}}{\delta(x_1)}.
\end{align*}
A further calculation gives
\begin{align*}%\label{v01x1x1x2x2}
|\partial_{x_1x_1}({\bf v}_0^3)^{(1)}|\leq \frac{C|x_1|^{l_1}}{\delta(x_1)},\quad |\partial_{x_2x_2}({\bf v}_0^3)^{(1)}|\leq \frac{C|x_1|^{l_1}}{\delta^2(x_1)},
\end{align*}
\begin{align*}%\label{v02x1x1x2x2}
|\partial_{x_1x_1}({\bf v}_0^3)^{(2)}|\leq \frac{C|x_1|^{l_1+1}}{\delta(x_1)},\quad |\partial_{x_2x_2}({\bf v}_0^3)^{(2)}|\leq \frac{C|x_1|^{l_1-1}}{\delta(x_1)}.
\end{align*}
Since  $\bar{p}_0^3=0$, then if $l_1\geq 2$,
\begin{align*}
|{\bf f}_0^3|=\left|\mu\,\Delta{\bf v}_0^3-\nabla\bar{p}_0^3\right|\leq\frac{C}{\delta(x_1)}\quad \mbox{in}~\Omega_{2R},
\end{align*} 
and   if $l_1=1$, we have 
\begin{align*}
|{\bf f}_0^3|=\left|\mu\,\Delta{\bf v}_0^3-\nabla\bar{p}_0^3\right|\leq\frac{C|x_1|}{\delta^2(x_1)}\quad \mbox{in}~\Omega_{2R}.
\end{align*} 
The estimates of $|{\bf f}_0^3|$ are  the same as that in \eqref{estdivv11p1} and \eqref{estf13}. By using the energy iteration technique as before, we derive the following result.

\begin{prop}\label{propv03}
Let ${\bf u}_0\in{C}^{2}(\Omega;\mathbb R^2),~p_0\in{C}^{1}(\Omega)$ be the solution to \eqref{equ_v3}. Then we have, in $\Omega_{R}$,
\begin{align*}
\|\nabla({\bf u}_0-{\bf v}_0^3)\|_{L^{\infty}(\Omega_{\delta/2}(x_1))}\leq \begin{cases}
	\frac{C}{\sqrt{\delta(x_1)}},&\quad l_1=1,\\
	C,&\quad l_1\geq2,
	\end{cases}
	\end{align*}
and 
\begin{equation*}
	\|\nabla^2({\bf u}_{0}-{\bf v}_{0}^3)\|_{L^{\infty}(\Omega_{\delta/2}(x_1))}+\|\nabla q_0^3\|_{L^{\infty}(\Omega_{\delta/2}(x_1))}\leq \begin{cases}
	\frac{C}{\delta^{3/2}(x_1)},&\quad l_1=1,\\
	\frac{C}{\delta(x_1)},&\quad l_1\geq2.
	\end{cases}
\end{equation*}
	Consequently, for $x\in\Omega_{R}$,
	\begin{align*}
	|\nabla {\bf u}_0|\leq
	\begin{cases}
	\frac{C}{\sqrt{\delta(x_1)}},&\quad l_1=1,\\
	C,&\quad l_1\geq2,
	\end{cases}
	\end{align*} and
	\begin{align*}
	\|p_0-(q_0^3)_{\Omega_R}\|_{L^{\infty}(\Omega_{\delta/2}(x_{1}))}\leq\begin{cases}
	\frac{C}{\varepsilon^{3/2}},&\quad l_1=1,\\
		\frac{C}{\varepsilon},&\quad l_1\geq 2.
	\end{cases}
	\end{align*}
\end{prop}

\begin{proof}[Proof of Theorem \ref{mainthm1}]   
Substituting the estimates of $Q_\beta[{\boldsymbol{\varphi}}]$ in Proposition \ref{lemaQb}, and the estimates of $|\nabla{\bf u}_0|$ in Proposition \ref{propv03} into Theorem \ref{mainthm}, we directly derive the estimates of $|\nabla{\bf u}|$ and $|p|$. Theorem \ref{mainthm1} is proved.
\end{proof}

\subsection{Proof of Theorem \ref{mainthm2}}\label{seccoro2}

In this subsection, we consider the boundary data ${\boldsymbol{\varphi}}\in{\bf\Phi}_{4}$. For $l_2=1$, we choose ${\bf v}_0^4\in C^{2}(\Omega;\mathbb R^2)$ satisfying, in $\Omega_{2R}$,
\begin{align}\label{v02.alpha}
{\bf v}_0^4=\boldsymbol\varphi\Big(\frac{1}{2}-k(x)\Big)+\begin{pmatrix}
\frac{-4x_1^{2}}{\delta(x_{1})}+\frac{1}{\kappa_1-\kappa}\\\\
G_0^4(x)
\end{pmatrix}
\Big(k^2(x)-\frac{1}{4}\Big),
\end{align} 
where $$G_0^4(x)=2x_1k(x)-\delta(x_1)\partial_{x_1}k(x)\left(\frac{-4x_1^{2}}{\delta(x_{1})}+\frac{1}{\kappa_1-\kappa}\right).$$

For $l_2\geq 2$, we seek ${\bf v}_0^4\in C^{2}(\Omega;\mathbb R^2)$ satisfying, in $\Omega_{2R}$,
\begin{align}\label{v02alpha}
{\bf v}_0^4=\boldsymbol\varphi\Big(\frac{1}{2}-k(x)\Big)+\begin{pmatrix}
\frac{-6x_1^{l_2+1}}{(l_2+1)\delta(x_{1})}\\\\
G_0^4(x)
\end{pmatrix}
\Big(k^2(x)-\frac{1}{4}\Big),
\end{align}
where $$G_0^4(x)=2x_1^{l_2}k(x)+6\partial_{x_1}k(x)\frac{x_1^{l_2+1}}{l_2+1}.$$

In order to control the biggest term, we take
\begin{align}\label{defp02}
\bar{p}_0^4=
\begin{cases}
\frac{2\mu }{\kappa_1-\kappa}\frac{ x_1 }{\delta^2(x_1)}+\mu \partial_{x_2}({\bf v}_0^4)^{(2)},&\quad l_2=1,\\
0,&\quad l_2\geq 2.
\end{cases}
\end{align}
By a direct calculation, we derive 
\begin{align*}
|\nabla{\bf v}_0^4|\leq 
\frac{Cx_1^{l_2-1}}{\delta(x_1)}.
\end{align*}
Combining with \eqref{defp02}, we deduce, in $\Omega_{2R}$,
\begin{align*}
|{\bf f}_0^4|=\left|\mu\,\Delta{\bf v}_0^4-\nabla\bar{p}_0^4\right|\leq
\begin{cases}
\frac{C}{\delta(x_1)},&\quad l_2=1,~l_2\geq3,\\
\frac{C|x_1|}{\delta^2(x_1)},&\quad l_2=2.
\end{cases}
\end{align*} 
With these estimates at hand, replicating the  proof of Proposition \ref{propu11},    we have 
\begin{prop}\label{propv04}
	Let ${\bf u}_0\in{C}^{2}(\Omega;\mathbb R^2),~p_0\in{C}^{1}(\Omega)$ be the solution to \eqref{equ_v3}. Then 
	\begin{align*}
	\|\nabla({\bf u}_0-{\bf v}_0^4)\|_{L^{\infty}(\Omega_{\delta/2}(x_1))}\leq 
	\begin{cases}
	C,&\quad l_2=1,l_2\ge3,\\
	\frac{C}{\sqrt{\delta(x_1)}},&\quad l_2=2,
	\end{cases}
	\end{align*}
and 
\begin{equation*}
	\|\nabla^2({\bf u}_{0}-{\bf v}_{0}^4)\|_{L^{\infty}(\Omega_{\delta/2}(x_1))}+\|\nabla q_0^4\|_{L^{\infty}(\Omega_{\delta/2}(x_1))}\leq \begin{cases}
			\frac{C}{\delta(x_1)},&\quad l_2=1, l_2\ge 3,\\
		\frac{C}{\delta^{3/2}(x_1)},&\quad l_2=2.
	\end{cases}
\end{equation*}
	Consequently,
	\begin{align*}
	|\nabla {\bf u}_0|\leq \begin{cases}
	\frac{C}{\delta(x_1)},&\quad l_2=1,\\
	\frac{C}{\sqrt{\delta(x_1)}},&\quad l_2=2,\\
	C,&\quad l_2\geq 3,
	\end{cases}
	\end{align*}
	and
	\begin{align*}
	\|p_0-(q_0^4)_{\Omega}\|_{L^{\infty}(\Omega_{\delta/2}(x_{1}))}\leq
	\begin{cases}
	\frac{C}{\varepsilon^{3/2}},&\quad l_2=1,2,\\
		\frac{C}{\varepsilon},&\quad l_2\geq 3,
	\end{cases}
	\quad x\in\Omega_{R}.
	\end{align*}
\end{prop}

\begin{proof}[Proof of Theorem  \ref{mainthm2}.]   
Substituting the estimates of $Q_\beta[{\boldsymbol{\varphi}}]$ in Proposition \ref{lemaQb}, and the estimates of $|\nabla{\bf u}_0|$ in Proposition \ref{propv04} into Theorem \ref{mainthm}, we directly derive the estimates of $|\nabla{\bf u}|$ and $|p|$. Theorem \ref{mainthm2} is proved.
\end{proof}

\section{Lower Bounds: Proof of Theorem \ref{mainthm02} and Theorem \ref{mainthmlower}}\label{sec6}

In this Section we give the lower bounds of $\nabla{\bf u}$ at the midpoint $(0,\varepsilon/2)$ of the shortest line between $\partial{D}_{1}$ and $\partial{D}$ to show the optimality of the blow-up rates obtained above, and prove Theorem \ref{mainthm02} for boundary data ${\boldsymbol{\varphi}}\in{\bf\Phi}_{i}$, $i=1,2$,  and Theorem \ref{mainthmlower} for ${\boldsymbol{\varphi}}\in{\bf\Phi}_{i}$, $i=3,4$.

\subsection{Proof of Theorem \ref{mainthm02}}\label{subphi1}
By virtue of \eqref{defaij}, we define 
\begin{align}\label{defQ1aij}
Q_{1,\beta}[{\boldsymbol\varphi}]:=Q_\beta[{\boldsymbol{\varphi}}]-a_{\beta1}=-\int_{\Omega} \left(2\mu e({\bf u}_{0}+{\bf u}_{1}),e({\bf u}_{\beta})\right)\mathrm{d}x,\quad\beta=1,2,3.
\end{align}
Then, from Lemma \ref{lema114D} and Lemma \ref{lema113D}, and Proposition \ref{lemaQb1}, one can see that $|Q_{1,\beta}[{\boldsymbol{\varphi}}]|\leq C$, $\beta=1,2,3$. From \eqref{ce}, it follows that 
\begin{align}\label{matrixC1}
\begin{pmatrix}
~a_{11}&a_{12}&a_{13}~\\\\
~a_{21}&a_{22}&a_{23}~\\\\
~a_{31}&a_{32}&a_{33}
\end{pmatrix}\begin{pmatrix}
~C^1-1~\\\\
C^2\\\\
C^3
\end{pmatrix}=\begin{pmatrix}
~Q_{1,1}[{\boldsymbol{\varphi}}]~\\\\
Q_{1,2}[{\boldsymbol{\varphi}}]\\\\
Q_{1,3}[{\boldsymbol{\varphi}}]
\end{pmatrix}.
\end{align}
Define the limit functional
\begin{align*}
Q^*_{1,\beta}[{\boldsymbol\varphi}]=-\int_{\Omega^0} \left(2\mu e({\bf u}^*_{0}+{\bf u}^*_{1}),e({\bf u}^*_{\beta})\right)\mathrm{d}x,\quad\beta=1,2,3,
\end{align*}
where $({\bf u}_\beta^*,p_\beta^*)$ and $({\bf u}_0^*,p_0^*)$ satisfy \eqref{defu*} and \eqref{maineqn touch}, respectively. 

Similarly, using \eqref{ce} again, 
\begin{align}\label{matrixC2}
\begin{pmatrix}
~a_{11}&a_{12}&a_{13}~\\\\
~a_{21}&a_{22}&a_{23}~\\\\
~a_{31}&a_{32}&a_{33}
\end{pmatrix}\begin{pmatrix}
~C^1~\\\\
C^2-1\\\\
C^3
\end{pmatrix}=\begin{pmatrix}
~Q_{2,1}[{\boldsymbol{\varphi}}]~\\\\
Q_{2,2}[{\boldsymbol{\varphi}}]\\\\
Q_{2,3}[{\boldsymbol{\varphi}}]
\end{pmatrix},
\end{align}
where
\begin{align}\label{defQ2aij}
Q_{2,\beta}[{\boldsymbol\varphi}]:=Q_\beta[{\boldsymbol{\varphi}}]-a_{\beta2}=-\int_{\Omega} \left(2\mu e({\bf u}_{0}+{\bf u}_{2}),e({\bf u}_{\beta})\right)\mathrm{d}x,\quad\beta=1,2,3.
\end{align}
Define 
\begin{align*}
Q^*_{2,\beta}[{\boldsymbol\varphi}]=-\int_{\Omega^0} \left(2\mu e({\bf u}^*_{0}+{\bf u}^*_{2}),e({\bf u}^*_{\beta})\right)\mathrm{d}x,\quad\beta=1,2,3.
\end{align*}

We are going to prove the convergence of $Q_{1,\beta}[{\boldsymbol{\varphi}}]$ and $Q_{2,\beta}[{\boldsymbol{\varphi}}]$ as $\varepsilon\rightarrow0$.
Define
$$\mathcal{C}_{r}:=\left\{x\in\mathbb R^{2}\big| |x_1|<r,~0\leq x_{2}\leq\varepsilon+\kappa_1r^2\right\},\quad r<R.$$

\begin{lemma}\label{lem difference v11}
	Let $({\bf u}_{1},p_{1})$ and $({\bf u}_{1}^*,p_{1}^{*})$ satisfy \eqref{equ_v1} and \eqref{defu*}, respectively. Then
	\begin{align}\label{converu1}
	|({\bf u}_{1}-{\bf u}_{1}^{*})(x)|\leq C\varepsilon^{1/2},\quad  x\in V\setminus \mathcal{C}_{\varepsilon^{1/4}},
	\end{align}
	where $C$ is a {\it universal constant} and $V:=D\setminus\overline{D_{1}\cup D_{1}^0}$. 
\end{lemma}

\begin{proof}
We first introduce an auxiliary function $k^*(x)$, satisfying $k^*(x)=\frac{1}{2}$ on $\partial D_1^0\setminus\{0\}$, $k^*(x)=-\frac{1}{2}$ on $\partial D$, especially,
\begin{equation}\label{defk*}
k^*(x)=\frac{x_2-\kappa x_1^2}{(\kappa_1-\kappa)x_1^2}-\frac{1}{2}\quad\hbox{in}\ \Omega_{2R}^{0},\quad \|k^{*}(x)\|_{C^{2}(\Omega^{0}\setminus\Omega_{R}^{0})}\leq\,C,
\end{equation}
where $\Omega_r^{0}:=\left\{(x_1,x_{2})\in \mathbb{R}^{2}~\big|~\kappa x_1^2<x_{2}<\kappa_1 x_1^2,~|x_1|<r\right\}$, $r<R$. 
Notice that $({\bf u}_{1}-{\bf u}_{1}^{*},p_1-p_1^{*})$ satisfies
\begin{equation*}
\begin{cases}
\nabla\cdot\sigma[{\bf u}_{1}^*-{\bf u}_{1}^*,p_{1}-p_{1}^*]=0,&\mathrm{in}~V:=D\setminus\overline{D_{1}\cup D_{1}^0},\\
\nabla\cdot ({\bf u}_{1}-{\bf u}_{1}^*)=0,&\mathrm{in}~V,\\
{\bf u}_{1}-{\bf u}_{1}^*={\boldsymbol\psi}_{1}-{\bf u}_{1}^*,&\mathrm{on}~\partial{D}_{1}\setminus D_{1}^0,\\
{\bf u}_{1}-{\bf u}_{1}^*={\bf u}_{1}-{\boldsymbol\psi}_{1},&\mathrm{on}~\partial{D}_{1}^0\setminus(D_{1}\cup\{0\}),\\
{\bf u}_{1}-{\bf u}_{1}^*=0,&\mathrm{on}~\partial{D}.
\end{cases}
\end{equation*} 
If $x\in\partial{D}_{1}\setminus D_{1}^0\subset\Omega^0\setminus\Omega_{R}^0$, then $(x_1,x_{2}-\varepsilon)\in\Omega^0\setminus\Omega_{R}^0$. By using the mean value theorem, 
\begin{align}\label{partial D11}
|({\bf u}_{1}-{\bf u}_{1}^{*})(x_1,x_{2})|&=|(\boldsymbol\psi_{1}-{\bf u}_{1}^*)(x_1,x_{2})|=|{\bf u}_{1}^*(x_1,x_{2}-\varepsilon)-{\bf u}_{1}^*(x_1,x_{2})|\leq C\varepsilon.
\end{align}

For $x\in\partial{D}_{1}^{0}\setminus(D_{1}\cup\mathcal{C}_{\varepsilon^{\theta}})$, where $0<\theta<1$ is some constant to be determined later, then 
we have
\begin{align}\label{partial D11*}
|({\bf u}_{1}-{\bf u}_{1}^*)(x_1,x_{2})|&=|({\bf u}_{1}-\boldsymbol\psi_{1})(x_1,x_{2})|=|{\bf u}_{1}(x_1,x_{2})-{\bf u}_{1}(x_1,x_{2}+\varepsilon)|\nonumber\\
&\leq\frac{C\varepsilon}{\varepsilon+(\kappa_1-\kappa)|x_1|^{2}}\leq C\varepsilon^{1-2\theta}.
\end{align}
Define ${\bf v}_1^*$ as another auxilliary function, replacing $k(x)$ in ${\bf v}_1$ by $k^*(x)$, then, for $x\in\Omega_{R}^{0}$ with $|x_1|=\varepsilon^{\theta}$, it follows from Proposition \ref{propu11} that
\begin{align}\label{partial x2 u11}
\left|\partial_{x_{2}}({\bf u}_{1}-{\bf u}_{1}^*)(x_1,x_{2})\right|
&=\left|\partial_{x_{2}}({\bf u}_{1}-{\bf v}_{1})+\partial_{x_{2}}({\bf v}_{1}-{\bf v}_{1}^*)
+\partial_{x_{2}}({\bf v}_{1}^*-{\bf u}_{1}^*)\right|(x_1,x_{2})\nonumber\\
&\leq C\left(1+\frac{1}{\varepsilon^{4\theta-1}}\right).
\end{align}
Thus, for $x\in\Omega_R^0$ whth $|x_1|=\varepsilon^{\theta}$, by using the triangle inequality, \eqref{partial D11}, \eqref{partial x2 u11}, and recalling ${\bf u}_{1}-{\bf u}_{1}^*=0$ on $\partial D$, we have   
\begin{align}\label{es v11*2D}
\left|({\bf u}_{1}-{\bf u}_{1}^*)(x_1,x_{2})\right|&\leq\left|({\bf u}_{1}-{\bf u}_{1}^*)(x_1,x_{2})-({\bf u}_{1}-{\bf u}_{1}^*)(x',\kappa x_1^2)\right|
+C\varepsilon^{1-2\theta}\nonumber\\
&\leq\left|\partial_{x_{2}}({\bf u}_{1}-{\bf u}_{1}^*)\right|\cdot(\kappa_1x_1^2-\kappa x^2_1)\Big|_{|x'|=\varepsilon^{\theta}}\nonumber\\
&\leq C\left(1+\frac{1}{\varepsilon^{4\theta-1}}\right)\varepsilon^{2\theta}=C\left(\varepsilon^{2\theta}+\varepsilon^{1-2\theta}\right).
\end{align}	
Letting $2\theta=1-2\theta$, we get $\theta=\frac{1}{4}$. Substituting it into \eqref{partial D11*}, \eqref{partial x2 u11} and \eqref{es v11*2D}, and recalling ${\bf u}_{1}-{\bf u}_{1}^{*}=0$ on $\partial D$, we obtain
\begin{align*}
|{\bf u}_{1}-{\bf u}_{1}^{*}|\leq C\varepsilon^{1/2}\quad\mbox{on}~\partial{(V\setminus\mathcal{C}_{\varepsilon^{1/4}})}.
\end{align*}
Applying Lemma \ref{lemmaxi}, we obtain \eqref{converu1} and Lemma \ref{lem difference v11} is proved.
\end{proof}

\begin{lemma}\label{lemmatildeQ}
For $\beta=1,3$, we have 
	\begin{equation*}
	\big|Q_{1,\beta}[{\boldsymbol\varphi}]-Q^*_{1,\beta}[{\boldsymbol\varphi}]\big|\leq C\varepsilon^{1/8},\quad {\boldsymbol{\varphi}}\in{\bf\Phi}_{1},
	\end{equation*}
	and 
	\begin{equation*}
	\big|Q_{2,\beta}[{\boldsymbol\varphi}]-Q^*_{2,\beta}[{\boldsymbol\varphi}]\big|\leq C\varepsilon^{1/8},\quad {\boldsymbol{\varphi}}\in{\bf\Phi}_{2}.
	\end{equation*}
\end{lemma}

\begin{proof}
	We estimate $\big|Q_{1,1}[{\boldsymbol\varphi}]-Q^*_{1,1}[{\boldsymbol\varphi}]\big|$ for example, since the other three cases are the same. Note that
	\begin{align*}%\label{defQaij}
	Q_{1,1}[{\boldsymbol\varphi}]&=-\int_{\Omega} \left(2\mu e({\bf u}_{0}+{\bf u}_{1}),e({\bf u}_{1})\right)\mathrm{d}x\\
	&=-\int_{\Omega_{\varepsilon^{1/8}}} \left(2\mu e({\bf u}_{0}+{\bf u}_{1}),e({\bf u}_{1})\right)\mathrm{d}x-\int_{\Omega\setminus\Omega_{\varepsilon^{1/8}}} \left(2\mu e({\bf u}_{0}+{\bf u}_{1}),e({\bf u}_{1})\right)\mathrm{d}x,
	\end{align*}
	and 
	\begin{align*}%\label{defQaij}
	Q^*_{1,1}[{\boldsymbol\varphi}]=-\int_{\Omega^0_{\varepsilon^{1/8}}} \left(2\mu e({\bf u}^*_{0}+{\bf u}^*_{1}),e({\bf u}^*_{1})\right)\mathrm{d}x-\int_{\Omega^0\setminus\Omega^0_{\varepsilon^{1/8}}} \left(2\mu e({\bf u}^*_{0}+{\bf u}^*_{1}),e({\bf u}^*_{1})\right)\mathrm{d}x.
	\end{align*}
	Using Proposition \ref{propu11} and \eqref{estu0u1}, we deduce
	\begin{align*}
	\left|-\int_{\Omega_{\varepsilon^{1/8}}} \left(2\mu e({\bf u}_{0}+{\bf u}_{1}),e({\bf u}_{1})\right)\mathrm{d}x\right|\leq C\int_{\Omega_{\varepsilon^{1/8}}}\frac{1}{\delta(x_1)}\mathrm{d}x\leq C\varepsilon^{1/8},
	\end{align*}
	and 
	\begin{align*}
	\left|-\int_{\Omega^0_{\varepsilon^{1/8}}} \left(2\mu e({\bf u}^*_{0}+{\bf u}^*_{1}),e({\bf u}^*_{1})\right)\mathrm{d}x\right|\leq C\int_{\Omega^0_{\varepsilon^{1/8}}}\frac{1}{(\kappa_1-\kappa)x_1^2}\mathrm{d}x\leq C\varepsilon^{1/8}.
	\end{align*}
	We thus obtain
	\begin{align}\label{tildeQ}
	Q_{1,1}[{\boldsymbol\varphi}]-Q^*_{1,1}[{\boldsymbol\varphi}]&=-\int_{\Omega\setminus\Omega_{\varepsilon^{1/8}}} \left(2\mu e({\bf u}_{0}+{\bf u}_{1}),e({\bf u}_{1})\right)\mathrm{d}x\nonumber\\
	&\quad+\int_{\Omega^0\setminus\Omega^0_{\varepsilon^{1/8}}} \left(2\mu e({\bf u}^*_{0}+{\bf u}^*_{1}),e({\bf u}^*_{1})\right)\mathrm{d}x+O(\varepsilon^{1/8}).
	\end{align}
	
	Next, we estimate the difference
	$$-\int_{\Omega\setminus\Omega_{\varepsilon^{1/8}}} \left(2\mu e({\bf u}_{0}+{\bf u}_{1}),e({\bf u}_{1})\right)\mathrm{d}x+\int_{\Omega^0\setminus\Omega^0_{\varepsilon^{1/8}}} \left(2\mu e({\bf u}^*_{0}+{\bf u}^*_{1}),e({\bf u}^*_{1})\right)\mathrm{d}x.$$ 
	In view of the boundedness of $|\nabla{\bf u}_0|$ and $|\nabla{\bf u}_1|$ in $D_1^0\setminus(D_1\cup\Omega_R)$ and $D_1\setminus D_1^0$,  and $|D_1^0\setminus(D_1\cup\Omega_R)|, |D_1\setminus D_1^0|\leq C\varepsilon$, we have 
	\begin{align}\label{estoutR}
	\int_{\Omega\setminus\Omega_{R}} \left(2\mu e({\bf u}_{0}+{\bf u}_{1}),e({\bf u}_{1})\right)\mathrm{d}x=\int_{D\setminus(D_1\cup D_1^0\cup\Omega_{R})} \left(2\mu e({\bf u}_{0}+{\bf u}_{1}),e({\bf u}_{1})\right)\mathrm{d}x+O(\varepsilon).
	\end{align}
	Recalling that
	$$\nabla\cdot\sigma[{\bf u}_{1}^*-{\bf u}_{1}^*,p_{1}-p_{1}^*]=0,\quad x\in D\setminus D_{1}\cup D_1^0\cup\Omega_{R},$$
	$$0\leq|{\bf u}_{1}|, |{\bf u}_{1}^{*}|\leq1,\quad x\in D\setminus D_{1}\cup D_1^0\cup\Omega_{R},$$
	and $\partial D_{1}, \partial D_1^0$, and $\partial D$ are $C^{2,\alpha}$, we have
	\begin{equation}\label{DD v1 1}
	|\nabla^{2}({\bf u}_{1}-{\bf u}_{1}^{*})|\leq|\nabla^{2}{\bf u}_{1}|+|\nabla^{2}{\bf u}_{1}^{*}|\leq C\quad\mbox{in}~D\setminus(D_{1}\cup D_1^0\cup\Omega_{R}).
	\end{equation}
	By using interpolation inequality, \eqref{DD v1 1}, and \eqref{converu1}, we obtain 
	\begin{equation}\label{D v1 *1}
	|\nabla({\bf u}_{1}-{\bf u}_{1}^{*})|\leq C\varepsilon^{1/2(1-\frac{1}{2})}=C\varepsilon^{1/4}\quad\mbox{in}~D\setminus(D_{1}\cup D_1^0\cup\Omega_{R}).
	\end{equation}
	Similar to the proof of \eqref{converu1}, we have 
	\begin{align*}%\label{converu0}
	|({\bf u}_{0}-{\bf u}_{0}^{*})(x)|\leq C\varepsilon^{1/2}\quad\mbox{in}~ V\setminus \mathcal{C}_{\varepsilon^{1/4}}.
	\end{align*}
	Then by the argument in deriving \eqref{D v1 *1}, we obtain
	\begin{equation}\label{D v0 *}
	|\nabla({\bf u}_{0}-{\bf u}_{0}^{*})|\leq C\varepsilon^{1/4}\quad\mbox{in}~D\setminus(D_{1}\cup D_1^0\cup\Omega_{R}).
	\end{equation}
	
It follows from \eqref{D v1 *1} and \eqref{D v0 *} that
	\begin{align*}
	&\int_{D\setminus(D_1\cup D_1^0\cup\Omega_{R})} \left(2\mu e({\bf u}_{0}+{\bf u}_{1}),e({\bf u}_{1})\right)\mathrm{d}x\\
	&=
	\int_{D\setminus(D_1\cup D_1^0\cup\Omega_{R})} \left(2\mu e({\bf u}^*_{0}+{\bf u}^*_{1}),e({\bf u}^*_{1})\right)\mathrm{d}x\\
	&\quad+\int_{D\setminus(D_1\cup D_1^0\cup\Omega_{R})} \left(2\mu e({\bf u}_{0}+{\bf u}_{1}-{\bf u}^*_{0}-{\bf u}^*_{1}),e({\bf u}_{1}-{\bf u}^*_{1})\right)\mathrm{d}x\\
	&\quad+\int_{D\setminus(D_1\cup D_1^0\cup\Omega_{R})} \left(2\mu e({\bf u}_{0}+{\bf u}_{1}-{\bf u}^*_{0}-{\bf u}^*_{1}),e({\bf u}^*_{1})\right)\mathrm{d}x\\
	&\quad-\int_{D\setminus(D_1\cup D_1^0\cup\Omega_{R})} \left(2\mu e({\bf u}^*_{0}+{\bf u}^*_{1}),e({\bf u}^*_{1}-{\bf u}_{1})\right)\mathrm{d}x\\
	&=\int_{\Omega^0\setminus\Omega^0_{R}} \left(2\mu e({\bf u}^*_{0}+{\bf u}^*_{1}),e({\bf u}^*_{1})\right)\mathrm{d}x+O(\varepsilon^{1/4}).
	\end{align*}
	Substituting it into \eqref{estoutR}, we obtain
	\begin{align}\label{estoutR00}
	\int_{\Omega\setminus\Omega_{R}} \left(2\mu e({\bf u}_{0}+{\bf u}_{1}),e({\bf u}_{1})\right)\mathrm{d}x=\int_{\Omega^0\setminus\Omega^0_{R}} \left(2\mu e({\bf u}^*_{0}+{\bf u}^*_{1}),e({\bf u}^*_{1})\right)\mathrm{d}x+O(\varepsilon^{1/4}).
	\end{align}
	
	Note that 
	\begin{align*}
	&\int_{\Omega_{R}\setminus\Omega_{\varepsilon^{1/8}}} \left(2\mu e({\bf u}_{0}+{\bf u}_{1}),e({\bf u}_{1})\right)\mathrm{d}x-\int_{\Omega^0_{R}\setminus\Omega^0_{\varepsilon^{1/8}}} \left(2\mu e({\bf u}^*_{0}+{\bf u}^*_{1}),e({\bf u}^*_{1})\right)\mathrm{d}x\\
	&=\int_{(\Omega_{R}\setminus\Omega_{\varepsilon^{1/8}})\setminus(\Omega_{R}^{0}\setminus\Omega_{\varepsilon^{1/8}}^{0})}\big(2\mu e({\bf u}_{0}+{\bf u}_{1}),e({\bf u}_{1})\big)\ dx\\
	&\quad+2\int_{\Omega_{R}^{0}\setminus\Omega_{\varepsilon^{1/8}}^{0}}\big(2\mu e({\bf u}_{0}+{\bf u}_{1}-{\bf u}^*_{0}-{\bf u}^*_{1}),e({\bf u}_{1}^*)\big)\ dx\nonumber\\
	&\quad-\int_{\Omega_{R}^{0}\setminus\Omega_{\varepsilon^{1/8}}^{0}}\big(2\mu e({\bf u}_{0}-{\bf u}^*_{0}),e({\bf u}_{1}^*)\big)\ dx-\int_{\Omega_{R}^{0}\setminus\Omega_{\varepsilon^{1/8}}^{0}}\big(2\mu e({\bf u}_{0}^*),e({\bf u}_{1}^*-{\bf u}_{1})\big)\ dx\nonumber\\
	&\quad+\int_{\Omega_{R}^{0}\setminus\Omega_{\varepsilon^{1/8}}^{0}}\big(2\mu e({\bf u}_{0}+{\bf u}_{1}-{\bf u}^*_{0}-{\bf u}^*_{1}),e({\bf u}_{1}-{\bf u}_{1}^{*})\big)\ dx\nonumber\\
	&=:\mbox{I}_{1}+\mbox{I}_{2}+\mbox{I}_{3}+\mbox{I}_{4}+\mbox{I}_{5}.
	\end{align*}
	For $\varepsilon^{1/8}\leq|z_{1}|\leq R$, by using the following change of variables:
	\begin{align*}
	\begin{cases}
	x_{1}-z_{1}=|z_{1}|^{2}y_{1},\\
	x_{2}=|z_{1}|^{2}y_{2},
	\end{cases}
	\end{align*}
	we rescale $\Omega_{|z_{1}|+|z_{1}|^{2}}\setminus\Omega_{|z_{1}|}$ into a nearly cube $Q_{1}$ in unit size, and $\Omega_{|z_{1}|+|z_{1}|^{2}}^{*}\setminus\Omega_{|z_{1}|}^{0}$ into $Q_{1}^{0}$.
	Let
	$${\bf U}_{1}={\bf u}_{1}(z_{1}+z_{1}^{2}y_{1},|z_{1}|^{2}y_{2}),\quad\mbox{in}~Q_{1},\quad\mbox{and}\quad {\bf U}_{1}^{*}={\bf u}_{1}^{*}(z_{1}+z_{1}^{2}y_{1},|z_{1}|^{2}y_{2}),\quad\mbox{in}~Q_{1}^{0}.$$
	Similar to \eqref{DD v1 1}, we get
	$$|\nabla^{2}{\bf U}_{1}|\leq C\quad\mbox{in}~Q_{1},\quad\mbox{and}~|\nabla^{2}{\bf U}_{1}^{*}|\leq C\quad\mbox{in}~Q_{1}^{0}.$$
	Using interpolation, we obtain
	$$|\nabla({\bf U}_{1}-{\bf U}_{1}^{*})|\leq C\varepsilon^{1/4}.$$
	Rescaling back to ${\bf u}_{1}-{\bf u}_{1}^{*}$, we have
	\begin{equation}\label{difference Dv1 1}
	|\nabla({\bf u}_{1}-{\bf u}_{1}^{*})|\leq C\varepsilon^{1/4}|x_{1}|^{-2}\quad\mbox{in}~\Omega_{R}^{0}\setminus\Omega_{\varepsilon^{1/8}}^{0}.
	\end{equation}

Similarly, we have
	\begin{equation}\label{Dv11}
	|\nabla {\bf u}_{1}|\leq C|x_{1}|^{-2}\quad\mbox{in}~\Omega_{R}\setminus\Omega_{\varepsilon^{1/8}},
	\end{equation}
	and
	\begin{equation}\label{Dv11*}
	|\nabla {\bf u}_{1}^{*}|\leq C|x_{1}|^{-2}\quad\mbox{in}~\Omega_{R}^{0}\setminus\Omega_{\varepsilon^{1/8}}^{0}.
	\end{equation}
	The estimates \eqref{difference Dv1 1}--\eqref{Dv11*} also hold for ${\bf u}_0$ and ${\bf u}_0^*$. 
	It follows from \eqref{Dv11} and $|(\Omega_{R}\setminus\Omega_{\varepsilon^{1/8}})\setminus(\Omega_{R}^{0}\setminus\Omega_{\varepsilon^{1/8}}^{0})|\leq C\varepsilon$ that
	\begin{align*}
	|\mbox{I}_{1}|\leq C\varepsilon\int_{\varepsilon^{1/8}<|x_{1}|\leq R}\frac{dx_{1}}{|x_{1}|^{2}}\leq C\varepsilon^{7/8}.
	\end{align*}
	Also, by using \eqref{difference Dv1 1} and \eqref{Dv11*}, we obtain
	\begin{align*}
	|\mbox{I}_{2}|+|\mbox{I}_{3}|+|\mbox{I}_{4}|\leq C\varepsilon^{1/4}\int_{\varepsilon^{1/8}<|x_{1}|\leq R}\frac{dx_{1}}{|x_{1}|^{2}}\leq C\varepsilon^{1/8},
	\end{align*}
	and by \eqref{difference Dv1 1}, we get
	\begin{align*}
	|\mbox{I}_{5}|\leq C\varepsilon^{1/2}\int_{\varepsilon^{1/8}<|x_{1}|\leq R}\frac{dx_{1}}{|x_{1}|^{2}}\leq C\varepsilon^{3/8}.
	\end{align*}
	We thus obtain 
	\begin{align}\label{outvare}
	\int_{\Omega_{R}\setminus\Omega_{\varepsilon^{1/8}}} \left(2\mu e({\bf u}_{0}+{\bf u}_{1}),e({\bf u}_{1})\right)\mathrm{d}x-\int_{\Omega^0_{R}\setminus\Omega^0_{\varepsilon^{1/8}}} \left(2\mu e({\bf u}^*_{0}+{\bf u}^*_{1}),e({\bf u}^*_{1})\right)\mathrm{d}x=O(\varepsilon^{1/8}).
	\end{align}
	Combining \eqref{tildeQ}, \eqref{estoutR00}, and \eqref{outvare}, we conclude
	$$\big|Q_{1,1}[{\boldsymbol\varphi}]-Q^*_{1,1}[{\boldsymbol\varphi}]\big|\leq C\varepsilon^{1/8}.$$
Lemma \ref{lemmatildeQ} is proved.
\end{proof}

Now we are in a position to prove Theorem \ref{mainthm02}.
\begin{proof}[Proof of Theorem \ref{mainthm02}.] 
	(1) For ${\boldsymbol{\varphi}}\in{\bf\Phi}_{1}$. From the definition of $k(x)$, \eqref{def_vx}, we have 
	$$k(0,\varepsilon/2)=0.$$
	Then, combining Proposition \ref{propu12}, Proposition \ref{propu13}, \eqref{estv3112}, and  \eqref{lamda0}, we have 
	\begin{equation*}
	\left|\sum_{\alpha=2}^{3}C^\alpha\partial_{x_2}{\bf u}_\alpha^{(1)}(0,\varepsilon/2)\right|\leq C.
	\end{equation*}
	Moreover, from \eqref{estv1112} and \eqref{v01v1}, one can verify that
	\begin{equation*}
	\partial_{x_2}{\bf v}_1^{(1)}(0,\varepsilon/2)=\frac{1}{\varepsilon},\quad \partial_{x_2}{\bf v}_{0}^{(1)}(0,\varepsilon/2)=-\frac{1}{\varepsilon}.
	\end{equation*}
	This together with Proposition \ref{propu11} and Proposition \ref{propu12} yields
	\begin{align*}
	|\partial_{x_2}{\bf u}_1^{(1)}(0,\varepsilon/2)+\partial_{x_2}{\bf u}_0^{(1)}(0,\varepsilon/2)|&\leq |\partial_{x_2}{\bf v}_1^{(1)}(0,\varepsilon/2)+\partial_{x_2}{\bf v}_0^{(1)}(0,\varepsilon/2)|\\
	&\quad+|\partial_{x_2}{\bf w}_1^{(1)}(0,\varepsilon/2)+\partial_{x_2}{\bf w}_0^{(1)}(0,\varepsilon/2)|\leq C.
	\end{align*}
	Thus,
	\begin{align}\label{Dulower}
	|\nabla{\bf u}(0,\varepsilon/2)|&=\left|\sum_{\alpha=1}^{3}C^\alpha\nabla{\bf u}_\alpha(0,\varepsilon/2)+\nabla{\bf u}_0(0,\varepsilon/2)\right|\nonumber\\
	&\geq \left|\sum_{\alpha=1}^{3}C^\alpha\partial_{x_2}{\bf u}_\alpha^{(1)}(0,\varepsilon/2)+\partial_{x_2}{\bf u}_0^{(1)}(0,\varepsilon/2)\right|\nonumber\\
	&\geq |C^1\partial_{x_2}{\bf u}_1^{(1)}(0,\varepsilon/2)+\partial_{x_2}{\bf u}_0^{(1)}(0,\varepsilon/2)|-C\nonumber\\
	&\geq |(C^1-1)\partial_{x_2}{\bf u}_1^{(1)}(0,\varepsilon/2)|-C\geq \frac{|C^1-1|}{\varepsilon}-C.
	\end{align}
	By using \eqref{matrixC1}, Cramer's rule, and \eqref{DetA}, we obtain
	\begin{align}\label{estC1-1}
	&C^1-1=\frac{1}{\det\mathbb{A}}\big(\mbox{cof}(\mathbb A)_{11} Q_{1,1}[{\boldsymbol\varphi}]-\mbox{cof}(\mathbb A)_{21} Q_{1,2}[{\boldsymbol\varphi}]+\mbox{cof}(\mathbb A)_{31} Q_{1,3}[{\boldsymbol\varphi}] \big)\nonumber\\
	&=\frac{(\kappa_1-\kappa)^{1/2}}{\mu \pi}\Big(Q_{1,1}[{\boldsymbol\varphi}]-(\kappa_1+\kappa) Q_{1,3}[{\boldsymbol\varphi}]\Big)\sqrt\varepsilon+O(1)\varepsilon.
	\end{align}
	If $Q_{1,1}^*[{\boldsymbol\varphi}]-(\kappa_1+\kappa) Q_{1,3}^*[{\boldsymbol\varphi}]\neq 0$, then by using Lemma \ref{lemmatildeQ}, there exists a small enough constant $\varepsilon_0>0$ such that for $0<\varepsilon<\varepsilon_0$,
	\begin{equation*}
	|Q_{1,1}[{\boldsymbol\varphi}]-(\kappa_1+\kappa) Q_{1,3}[{\boldsymbol\varphi}]|\geq \frac{1}{2}|Q_{1,1}^*[{\boldsymbol\varphi}]-(\kappa_1+\kappa) Q_{1,3}^*[{\boldsymbol\varphi}]|>0.
	\end{equation*}
	Thus, from \eqref{estC1-1}, we have 
	\begin{align*}%\label{lowerC1}
	|C^1-1|\geq\frac{|Q_{1,1}^*[{\boldsymbol\varphi}]-(\kappa_1+\kappa) Q_{1,3}^*[{\boldsymbol\varphi}]|}{C}\sqrt\varepsilon.
	\end{align*}
	This together with \eqref{Dulower} yields
	\begin{align*}
	|\nabla{\bf u}(0,\varepsilon/2)|\geq\frac{1}{C\sqrt\varepsilon}.
	\end{align*}
	
(2) For ${\boldsymbol{\varphi}}\in{\bf\Phi}_{2}$. From \eqref{estv1112}, \eqref{estv2112} and \eqref{estv3112}, we have
	\begin{align*}
	\partial_{x_2}{\bf v}_1^{(1)}(0,\varepsilon/2)=\frac{1}{\varepsilon}, \quad\partial_{x_2}{\bf v}_2^{(1)}(0,\varepsilon/2)=\partial_{x_2}{\bf v}_0^{(1)}(0,\varepsilon/2)=0,\quad\partial_{x_2}{\bf v}_3^{(1)}(0,\varepsilon/2)=2.
	\end{align*}
	Then, by virtue of Proposition \ref{propu11}, Proposition \ref{propu12}, Proposition \ref{propu13} and \eqref{lamda20}, we deduce
	\begin{align}\label{Dulower2}
	|\nabla{\bf u}(0,\varepsilon/2)|&=\left|\sum_{\alpha=1}^{3}C^\alpha\nabla{\bf u}_\alpha(0,\varepsilon/2)+\nabla{\bf u}_0(0,\varepsilon/2)\right|\nonumber\\
	&\geq \left|\sum_{\alpha=1}^{3}C^\alpha\partial_{x_2}{\bf u}_\alpha^{(1)}(0,\varepsilon/2)+\partial_{x_2}{\bf u}_0^{(1)}(0,\varepsilon/2)\right|\nonumber\\
	&\geq |C^1\partial_{x_2}{\bf u}_1^{(1)}(0,\varepsilon/2)+C^2\partial_{x_2}{\bf u}_2^{(1)}(0,\varepsilon/2)+\partial_{x_2}{\bf u}_0^{(1)}(0,\varepsilon/2)|-C\nonumber\\
	&\geq \frac{|C^1|}{\varepsilon}-|\nabla({\bf w}_{2}+{\bf w}_{0})|-C.
    \end{align}	
    It follows from  \eqref{equ_v1} and \eqref{equ_v3} that ${\bf u}_2+{\bf u}_0$  satisfies
	\begin{equation*}
	\begin{cases}
	\nabla\cdot\sigma[{\bf u}_2+{\bf u}_0,p_{2}+p_{0}]=0,\quad\nabla\cdot ({\bf u}_{2}+{\bf u}_{0})=0&\mathrm{in}~\Omega,\\
	{\bf u}_{2}+{\bf u}_{0}={\boldsymbol{\psi}_2}&\mathrm{on}~\partial{D}_{1},\\
	{\bf u}_{2}+{\bf u}_{0}={\boldsymbol{\varphi}}&\mathrm{on}~\partial{D}.
	\end{cases}
	\end{equation*}
	Since ${\boldsymbol{\varphi}}={\boldsymbol{\psi}_2}$ in $\Omega_{2R}$, 
	 using the method in proving Proposition \ref{propu11}, we have \begin{equation}\label{w2+w0}
	|\nabla({\bf w}_{2}+{\bf w}_{0})|\leq C.
	\end{equation}
	Moverover, by using \eqref{matrixC2}, the Cramer's rule, and \eqref{DetA}, we obtain
	\begin{align}\label{estC11}
	C^1&=\frac{1}{\det\mathbb{A}}\big(\mbox{cof}(\mathbb A)_{11} Q_{2,1}[{\boldsymbol\varphi}]-\mbox{cof}(\mathbb A)_{21} Q_{2,2}[{\boldsymbol\varphi}]+\mbox{cof}(\mathbb A)_{31} Q_{2,3}[{\boldsymbol\varphi}] \big)\nonumber\\
	&=\frac{(\kappa_1-\kappa)^{1/2}}{\mu \pi}\Big(Q_{2,1}[{\boldsymbol\varphi}]-(\kappa_1+\kappa) Q_{2,3}[{\boldsymbol\varphi}]\Big)\sqrt\varepsilon+O(1)\varepsilon.
	\end{align}
	Similarly, if $Q_{2,1}^*[{\boldsymbol\varphi}]-(\kappa_1+\kappa) Q_{2,3}^*[{\boldsymbol\varphi}]\neq 0$, by using Lemma \ref{lemmatildeQ} and \eqref{Dulower2}--\eqref{estC11}, we have 
	\begin{align*}
	|\nabla{\bf u}(0,\varepsilon/2)|\geq\frac{1}{C\sqrt\varepsilon}.
	\end{align*}
		Theorem \ref{mainthm02} is proved.		
\end{proof} 

\subsection{Proof of Theorem \ref{mainthmlower}}
	\begin{lemma}\label{zydyl}
		For $\beta=1,2,3$, we have 
		\begin{equation*}
		\big|Q_{\beta}[{\boldsymbol\varphi}]-Q^*_{\beta}[{\boldsymbol\varphi}]\big|\leq C\varepsilon^{1/8},
		\end{equation*}
		where ${\boldsymbol{\varphi}}\in{\bf\Phi}_{i}$ which is defined in  \eqref{defphi3}, $i=3,4$, $l_1\geq 2$, and $l_2\geq 3$.
	\end{lemma}
	
	\begin{proof}
We follow the idea in \cite{BJL,LX0} and prove the case of $i=3$ and $\beta=1$ for instance, since other cases are the same.	Recalling the definition of $Q_{1}[{\boldsymbol\varphi}]$ in \eqref{aijbj}, by using the integration by parts, we have 
		\begin{equation*}
		Q_{1}[{\boldsymbol\varphi}]=-\int_{\partial D} {\boldsymbol\varphi}\cdot\sigma[{\bf u}_1,p_1-(q_1)_{\mathcal C_{out}}]\nu,
		\end{equation*}
		where $\mathcal C_{out}:=V\setminus \mathcal{C}_{\varepsilon^{1/4-\tilde{\gamma}}}$, $q_1=p_1-\bar p_1$, and $\bar p_1$ is defined in \eqref{defp113D}. 
		Similarly,
		\begin{equation*}
		Q^*_{1}[{\boldsymbol\varphi}]=-\int_{\partial D} {\boldsymbol\varphi}\cdot\sigma[{\bf u}_1^*,p_1^*-(q_1^*)_{\mathcal C_{out}}]\nu,
		\end{equation*}
		where $({\bf u}_1^*,p_1^*)$ is defined in \eqref{defu*} and $q_1^{*}=p_1^{*}-\bar p_1^{*}$.
		Thus, we have
		\begin{equation*}%\label{111111}
		Q_{1}[{\boldsymbol\varphi}]-Q^*_{1}[{\boldsymbol\varphi}]=-\int_{\partial D}{\boldsymbol\varphi}\cdot\sigma[{\bf u}_1-{\bf u}_1^{*},p_1-(q_1)_{\mathcal C_{out}}-p_1^{*}+(q_1^*)_{\mathcal C_{out}}]\nu.
		\end{equation*}
		In view of \eqref{converu1} and the  standard boundary  estimates for Stokes systems with ${\bf u}_{1}-{\bf u}_{1}^{*}=0$ on $\partial{D}$ (see, for example \cite{Kratz,Mazya}), we obtain for any $0<\tilde\gamma<\frac{1}{4}$, 
		\begin{align*}
		|\nabla({\bf u}_{1}-{\bf u}_{1}^{*})(x)|+|(p_1-(q_1)_{\mathcal C_{out}}-p_1^{*}+(q_1^*)_{\mathcal C_{out}})(x)|\leq C\varepsilon^{\tilde{\gamma}},\quad  x\in V\setminus \mathcal{C}_{\varepsilon^{1/4-\tilde{\gamma}}},
		\end{align*}
		which implies that 
		\begin{equation}\label{estBout}
		\left|\int_{\partial D\setminus\mathcal{C}_{\varepsilon^{1/4-\tilde{\gamma}}}}{\boldsymbol\varphi}\cdot\sigma[{\bf u}_1-{\bf u}_1^{*},p_1-(q_1)_{\mathcal C_{out}}-p_1^{*}+(q_1^*)_{\mathcal C_{out}}]\nu\right|\leq C\varepsilon^{\tilde{\gamma}}.
		\end{equation}
		
		Next we estimate 
		\begin{align*}
		\mathcal{B}&:=\int_{\partial D\cap\mathcal{C}_{\varepsilon^{1/4-\tilde{\gamma}}}}{\boldsymbol\varphi}\cdot\sigma[{\bf u}_1-{\bf u}_1^{*},p_1-(q_1)_{\mathcal C_{out}}-p_1^{*}+(q_1^*)_{\mathcal C_{out}}]\nu\\
		&=\int_{\partial D\cap\mathcal{C}_{\varepsilon^{1/4-\tilde{\gamma}}}}{\boldsymbol\varphi}\cdot\sigma[{\bf v}_1-{\bf v}_1^{*},\bar p_1-\bar p_1^{*}]\nu\\
		&\quad+\int_{\partial D\cap\mathcal{C}_{\varepsilon^{1/4-\tilde{\gamma}}}}{\boldsymbol\varphi}\cdot\sigma[{\bf w}_1-{\bf w}_1^{*},q_1-(q_1)_{\mathcal C_{out}}-q_1^{*}+(q_1^*)_{\mathcal C_{out}}]\nu=:\mathcal{B}_{v}+\mathcal{B}_{w},
		\end{align*}
		where ${\bf w}_1={\bf u}_1-{\bf u}_1$ and ${\bf w}_1^{*}={\bf u}_1^{*}-{\bf v}_1^{*}$. Using ${\boldsymbol\varphi}=(x_1^{l_1},0)^{\mathrm T}$ in $\Omega_{2R}$ with $l_1\geq2$, we have 
		\begin{align*}
		\mathcal{B}_{v}&=\int_{\partial D\cap\mathcal{C}_{\varepsilon^{1/4-\tilde{\gamma}}}}x_1^{l_1}\big(n_1(2\mu\partial_{x_1}({\bf v}_1-{\bf v}_1^*)^{(1)}-\bar p_1+\bar p_1^{*})\\
		&\qquad+n_2(\partial_{x_1}({\bf v}_1-{\bf v}_1^*)^{(2)}+\partial_{x_2}({\bf v}_1-{\bf v}_1^*)^{(1)})\big).
		\end{align*}
		Recall that ${\bf v}_1^*$ and $\bar p_1^*$ are defined similarly as in \eqref{v1alpha} and \eqref{defp113D} with $k^*(x)$ in place of $k(x)$, respectively, where $k^*(x)$ is defined in \eqref{defk*}. Combining \eqref{estv112}--\eqref{estv11223} and 
		$$n_1=\frac{2\kappa x_1}{\sqrt{1+4\kappa^2x_1^2}},\quad n_2=\frac{-1}{\sqrt{1+4\kappa^2x_1^2}},$$
		we have 
		\begin{align*}
		|\mathcal{B}_{v}|\leq \int_{\partial D\cap\mathcal{C}_{\varepsilon^{1/4-\tilde{\gamma}}}}\frac{C|x_1|^{l_1+2}}{|x_1|^4}\leq C\varepsilon^{1/4-\tilde{\gamma}}.
		\end{align*}
		By Proposition \ref{propu11}, we obtain
		$$|\mathcal{B}_{w}|\leq C\varepsilon^{1/4-\tilde{\gamma}}.$$
		Thus,
		\begin{equation}\label{estBw}
		|\mathcal{B}|\leq C\varepsilon^{1/4-\tilde{\gamma}}.
		\end{equation}
		Taking $\tilde{\gamma}=\frac{1}{8}$, we obtain from \eqref{estBout} and \eqref{estBw} that
		\begin{equation*}
		| Q_{1}[{\boldsymbol\varphi}]-Q^*_{1}[{\boldsymbol\varphi}]|\leq C\varepsilon^{1/8}.
		\end{equation*}
		The proof is completed.   
	\end{proof}
	
	\begin{proof}[Proof of Theorem \ref{mainthmlower}.]
(1)		If $l_1\geq 2$, then we obtain from Proposition \ref{propv03} that
		$$|\nabla{\bf u}_0|\leq C\quad\mbox{in}~\Omega.$$ 
		Proposition \ref{propu12}, together with \eqref{estv3112}, and  \eqref{lamda0} yields 
		\begin{equation*}
		\left|\sum_{\alpha=2}^{3}C^\alpha\partial_{x_2}{\bf u}_\alpha^{(1)}(0,\varepsilon/2)\right|\leq C\sqrt\varepsilon|\ln\varepsilon|,
		\end{equation*}
		where we used $k(0,\varepsilon/2)=0$. Then 
		\begin{align}\label{lowerDu}
		|\nabla{\bf u}(0,\varepsilon/2)|&=\left|\sum_{\alpha=1}^{3}C^\alpha\nabla{\bf u}_\alpha(0,\varepsilon/2)+\nabla{\bf u}_0(0,\varepsilon/2)\right|\nonumber\\
		&\geq \left|\sum_{\alpha=1}^{3}C^\alpha\partial_{x_2}{\bf u}_\alpha^{(1)}(0,\varepsilon/2)+\partial_{x_2}{\bf u}_0^{(1)}(0,\varepsilon/2)\right|\nonumber\\
		&\geq |C^1\partial_{x_2}{\bf u}_1^{(1)}(0,\varepsilon/2)|-C.
		\end{align}
		From \eqref{estv1112} and Proposition \ref{propu11}, one can verify that
		\begin{equation*}
		|\partial_{x_2}{\bf u}_1^{(1)}(0,x_2)|=|\partial_{x_2}{\bf v}_1^{(1)}(0,x_2)+\partial_{x_2}{\bf w}_1^{(1)}(0,x_2)|\geq \frac{1}{C\varepsilon}.
		\end{equation*}
		Then \eqref{lowerDu} becomes
		\begin{align}\label{Dulower1}
		|\nabla{\bf u}(0,\varepsilon
		/2)|\geq \frac{|C^1|}{C\varepsilon}-C.
		\end{align}
		
We proceed to prove the lower bound of $C^1$. From \eqref{estC1} and \eqref{DetA}, it follows that
		\begin{align*}
		C^1=\frac{(\kappa_1-\kappa)^{1/2}\sqrt{\varepsilon}}{\mu\,\pi}\left(Q_{1}[{\boldsymbol\varphi}]-(\kappa_1+\kappa)Q_{3}[{\boldsymbol\varphi}]\right)+O(1)\varepsilon^2.
		\end{align*}
		If $Q_{1}^*[{\boldsymbol\varphi}]-(\kappa_1+\kappa)Q_{3}^*[{\boldsymbol\varphi}]\neq 0$ for some $l_1\geq 2$, then by using Lemma \ref{zydyl}, there exists a small enough constant $\varepsilon_0>0$ such that for $0<\varepsilon<\varepsilon_0$,
		\begin{equation*}
		|Q_{1}[{\boldsymbol\varphi}]-(\kappa_1+\kappa)Q_{3}[{\boldsymbol\varphi}]|\geq \frac{1}{2}|Q_{1}^*[{\boldsymbol\varphi}]-(\kappa_1+\kappa)Q_{3}^*[{\boldsymbol\varphi}]|>0.
		\end{equation*}
		Thus, 
		\begin{align*}%\label{lowerC1}
		|C^1|
		&\geq\frac{|Q_{1}^*[{\boldsymbol\varphi}]-(\kappa_1+\kappa)Q_{3}^*[{\boldsymbol\varphi}]|}{C}\sqrt\varepsilon.
		\end{align*}
		This, in combination with  \eqref{Dulower1}, we deduce
		\begin{align*}
		|\nabla{\bf u}(0,\varepsilon
		/2)|\geq\frac{1}{C\sqrt{\varepsilon}}.
		\end{align*} 
		
	(2)  If $l_2\geq 3$, then by using Proposition \ref{propu12} and \eqref{lamda20}, the estimate \eqref{Dulower1} also holds. Similarly, if $Q_{1}^*[{\boldsymbol\varphi}]-(\kappa_1+\kappa)Q_{3}^*[{\boldsymbol\varphi}]\neq 0$ for some $l_2\geq 3$, then there exists a sufficiently small constant $\varepsilon_0>0$ such that for $0<\varepsilon<\varepsilon_0$,
	\begin{equation*}
	|Q_{1}[{\boldsymbol\varphi}]-(\kappa_1+\kappa)Q_{3}[{\boldsymbol\varphi}]|\geq \frac{1}{2}|Q_{1}^*[{\boldsymbol\varphi}]-(\kappa_1+\kappa)Q_{3}^*[{\boldsymbol\varphi}]|>0.
	\end{equation*}
	Thus, 
	\begin{align*}
	|\nabla{\bf u}(0,\varepsilon
	/2)|\geq\frac{1}{C\sqrt{\varepsilon}}.
	\end{align*} 
	Theorem \ref{mainthmlower} is proved
	\end{proof}

\noindent{\bf{\large Acknowledgements.}}
H.G. Li was partially supported by NSF of China (11971061).

\end{document}